\documentclass[12pt]{amsart}
\usepackage{amsmath}
\usepackage{amssymb}
\usepackage{epsfig}
\usepackage{graphicx}
\usepackage{framed}
\usepackage{color}
\definecolor{shadecolor}{gray}{0.875}

\numberwithin{equation}{section}

\input xy
\xyoption{all}

\calclayout
\allowdisplaybreaks[3]

\theoremstyle{plain}
\newtheorem{prop}{Proposition}[section]

\newtheorem{theo}[prop]{Theorem}
\newtheorem{coro}[prop]{Corollary}

\newtheorem{lemm}[prop]{Lemma}

\theoremstyle{definition}
\newtheorem{defi}[prop]{Definition}

\newtheorem{ques}[prop]{Question}
\newtheorem{conj}[prop]{Conjecture}

\newtheorem*{mconj}{Manin's Conjecture}

\newtheorem*{babconj}{Conjecture WBAB$_{n}$}

\newtheorem{rema}[prop]{Remark}

\newtheorem{exam}[prop]{Example}

\def\ra{\rightarrow}

\def\cY{{\mathcal Y}}

\def\rk{{\mathrm{rk}}}

\def\bN{{\mathbb N}}
\def\bP{{\mathbb P}}
\def\bQ{{\mathbb Q}}
\def\bR{{\mathbb R}}
\def\bZ{{\mathbb Z}}

\def\eff{\mathrm{eff}}
\def\Pic{\mathrm{Pic}}

\def\NS{\mathrm{NS}}
\def\vol{\mathrm{vol}}
\def\Chow{\mathrm{Chow}}
\def\Mor{\mathrm{Mor}}
\def\Nef{\mathrm{Nef}}

\makeatother
\makeatletter

\author{Brian Lehmann}
\address{Department of Mathematics \\
Boston College  \\
Chestnut Hill, MA \, \, 02467}
\email{lehmannb@bc.edu}

\author{Sho Tanimoto }
\address{Department of Mathematics\\
Rice University, MS 136 \\
Houston, Texas  77251-1892 \\
USA}
\email{sho.tanimoto@rice.edu}

\author{Yuri Tschinkel}
\address{Courant Institute\\
                New York University \\
                New York, NY 10012 \\
                USA }
\address{Simons Foundation\\
                 160 Fifth Av.\\ 
                 New York, NY 10010}                
\email{tschinkel@cims.nyu.edu}

\title[Balanced line bundles]{Balanced line bundles on Fano varieties}

\begin{document}
\date{\today}

\begin{abstract}
A conjecture of Batyrev and Manin relates arithmetic properties of varieties with 
ample anticanonical class to geometric invariants; in particular, 
counting functions defined by metrized ample line bundles and the corresponding 
asymptotics of rational points of bounded height are interpreted 
in terms of cones of effective divisors and 
certain thresholds with respect to these cones. This framework leads to 
the notion of balanced line bundles, whose counting functions, conjecturally, 
capture generic distributions of rational points. 
We investigate balanced line bundles in the context of
the Minimal Model Program, with special regard to the 
classification of Fano threefolds.  
\end{abstract}

\maketitle

\section{Introduction} 
\label{secct:intro}

Let $X$ be a smooth projective variety defined over a number field $F$ 
and $\mathcal L=(L, \|\cdot\|)$ an ample, adelically metrized, line bundle on $X$. 
Such line bundles give rise to height functions
$$
\begin{array}{rcl}
X(F) & \ra &  \bR_{>0} \\
x & \mapsto & H_{\mathcal L }(x)
\end{array}
$$ 
on the set of $F$-rational points (see, e.g.,  \cite[Section 1.3]{CL-T-torsors} for the definitions). 
A basic result is that the associated 
counting function
$$
N(X(F),\mathcal L, B):=\#\{ x\in X(F) \, |\, H_{\mathcal L}(x)\le B\} 
$$
is finite, for each $B\in \bR$.   
Conjectures of Manin and Batyrev-Manin concern the asymptotic behavior of $N(X,\mathcal L, B)$, as $B\ra \infty$, 
for Fano varieties, i.e., varieties $X$ with ample anticanical class $-K_X$. The conjectures predict that
this asymptotic is controlled by the geometry of $X$ and $L$ \cite{BM}. 
More precisely, define the geometric constants
\begin{equation*}
a(X,L) = \min \{ t\in \bR \mid t[L] + [K_X] \in \Lambda_{\mathrm{eff}}(X) \}.
\end{equation*}
and
\begin{align*}
b(X,L) = & \textrm{ the codimension of the minimal} \\
& \textrm{supported face of }  \Lambda_{\mathrm{eff}}(X) \textrm{ containing} \\
& \textrm{the numerical class } a(X, L)[L] + [K_X].
\end{align*}
The following extension of Manin's original conjecture builds on \cite{BM}, \cite{Peyre}, \cite{BT}:

\begin{mconj}
Let $X$ be a smooth projective variety over a number field, with ample anticanonical class $-K_X$.  
Let $\mathcal{L} = (L,\Vert \cdot \Vert)$ be an ample adelically metrized line bundle on $X$. 
There exists a Zariski open set $X^{\circ}\subseteq X$ such that 
for every sufficiently large finite extension $F'$ of the ground field $F$, one has
\begin{equation}
\label{eqn:manin}
N(X^\circ(F'),\mathcal L, B)
\sim c(X^\circ,\mathcal{L})B^{a(X,L)} \log(B)^{b(X,L)-1}, \quad B\ra \infty,
\end{equation}
for some constant $c(X^\circ,\mathcal{L})>0$.
\end{mconj}

Implicitly this conjecture anticipates a compatibility between the constants $a(X,L), b(X,L)$ and the constants $a(Y,L), b(Y,L)$ for subvarieties $Y$ of $X$: any subvariety $Y$ with $Y \cap X^\circ \neq \emptyset$ must satisfy
\begin{equation}
\label{eqn:fail}
(a(Y,L),b(Y,L)) \leq (a(X,L),b(X,L))
\end{equation}
in the lexicographic order.  
In particular, to construct a suitable open set $X^{\circ}$ we have to  
remove $L$-accumulating subvarieties of $X$, i.e.,  
subvarieties $Y\subset X$ which fail inequality \eqref{eqn:fail}.

There is a large body of work proving Manin's conjecture for various classes of varieties, 
most of which are either hypersurfaces of low degree,  
Del Pezzo surfaces, or equivariant compactifications of homogeneous spaces of 
linear algebraic groups (see \cite{T} for a survey and further references). While there are still many 
classes of equivariant compactifications for which Manin's conjecture is open, e.g., 
compactifications of solvable groups, or of extensions of semisimple groups by unipotent groups, we have 
the following geometric result (see Section~\ref{sect:examples}):

\begin{theo}
\label{thm:first} 
Let $X$ be a generalized flag variety
and $L$ an ample line bundle on $X$.  
Then $a(Y,L) \leq a(X,L)$ for every subvariety $Y$.
\end{theo}

However, Manin's Conjecture is known to fail for certain Fano hypersurfaces 
in products of projective spaces \cite{BT-cubic}. 
In fact, all known failures of Manin's Conjecture are explained by the geometric incompatibility discussed 
above, i.e., the possibility that $L$-accumulating subvarieties are Zariski dense 
(see Examples \ref{exam:maninconjfail} and \ref{exam:tamagawaconstantfail}).

In this paper, we give a systematic analysis of the geometric invariants $a$ and $b$, 
building on \cite{BT} and \cite{balanced}, and apply it to Fano threefolds.  
We will mostly work over an algebraically closed field of characteristic zero, to focus on the underlying geometry.  
Our first general result is:

\begin{theo} 
\label{theo:weakmanincompatibility}
Let $X$ be a smooth projective variety and $L$ an ample line bundle on $X$.  
There is a countable union $V$ of proper closed subsets of $X$ such that every subvariety $Y\subset X$ satisfying
\begin{equation*}
a(Y,L) > a(X,L)
\end{equation*}
is contained in $V$.
\end{theo}

Although the countability of $V$ is necessary in general, Manin's conjecture predicts a stronger statement: for uniruled varieties $X$, the subset $V$ in Theorem \ref{theo:weakmanincompatibility} should be Zariski closed.  
We prove this, assuming 
boundedness of terminal $\mathbb{Q}$-Fano varieties of Picard number $1$  as predicted by the Borisov-Alexeev-Borisov Conjecture (WBAB$_{n-1}$, see Section \ref{sect:rigidity} for a precise formulation): 

\begin{theo} \label{theo:manincompatibility}
Assume Conjecture WBAB$_{n-1}$.  Let $X$ be a smooth uniruled projective variety of dimension $n$ and 
$L$ an ample divisor on $X$.  Then there exists a proper closed subset $V \subset X$ 
such that every subvariety 
$Y\subset X$ satisfying $a(Y,L) > a(X,L)$ is contained in $V$.

In particular, the statement holds for $X$ of dimension at most $4$.
\end{theo}

These theorems demonstrate that the geometric behavior of the constant $a$ is 
compatible with Manin's Conjecture.  In fact, there are no known counterexamples to the weaker conjecture
\begin{equation*}
N(X^{\circ}(F), \mathcal L, B) \ll B^{a(X,L)+\epsilon}, \quad \epsilon>0, \quad B\ra \infty. 
\end{equation*}

In contrast, the properties of the constant $b$ are more subtle: 
it often increases on subvarieties. This 
geometric behavior can have number-theoretic consequences; 
see \cite{BT-cubic} as well as Section~\ref{sect:applications} 
for examples. These considerations motivated the introduction of 
{\em balanced line bundles} \cite{balanced}:

\begin{defi}
Let $X$ be a smooth uniruled projective variety and $L$ an ample line bundle on $X$.  
We say that the pair $(X,L)$ (or just $L$ if $X$ is understood) is weakly balanced 
if there is a proper closed subset $V \subset X$ such that for every 
$Y \not \subset V$ we have 
$$
a(Y,L) \leq a(X,L)\text{ and}  \text{ if } a(Y, L) = a(X,L), \text{ then } b(Y,L) \leq b(X,L).
$$
We say that $(X,L)$ is balanced if there is a proper closed subset 
$V$ such that for $Y \not \subset V$ we have the stronger condition
\begin{itemize}
\item $a(Y,L) < a(X,L)$, or
\item $a(Y,L) = a(X,L)$ and $b(Y,L) < b(X,L)$.
\end{itemize}
In either case, we call the subset $V$ exceptional, or accumulating.  
\end{defi}

In Section~\ref{sect:balanced}, 
we will also introduce notions of weakly $a$-balanced 
and $a$-balanced line bundles (see Definition~\ref{defi: a-balanced})
and discuss their properties.

An effective $\bQ$-divisor $D$ on $X$ is called rigid
if $H^0(X,mD)=1$ for all sufficiently divisible $m\in \bN$. 
The following result is closely related to conjectures of \cite{BT}: 

\begin{theo}
\label{thm:rigid}
Let $X$ be a smooth uniruled projective variety and $L$ an ample divisor on $X$.  
If $(X,L)$ is balanced, then $K_{X} + a(X,L)L$ is numerically equivalent to a 
rigid effective divisor. 
In particular, $X$ is birational to a log-Fano variety and $K_{X} + a(X,L)L$ 
is numerically equivalent to an exceptional divisor for a birational map.
\end{theo}

In the second half of the paper we turn to examples.  
We perform an in-depth analysis of the balanced property of $-K_X$ for (most) 
primitive Fano threefolds, following the classification theory of
\cite{FanoI}, \cite{FanoII}, \cite{iskov79}, and \cite{mori-mukai}.
Our main results are Theorems~\ref{thm:fano-rank1} and \ref{thm:fano-rank2}, determining the balanced properties of $-K_X$ 
for primitive Fano threefolds of Picard ranks 1 and 2.

\

Here is the roadmap of the paper: In Section~\ref{sect:preliminaries} we recall basic notions of the Minimal Model Program which will be relevant for the analysis of balanced line bundles, introduced in Section~\ref{sect:balanced}. 
In Section~\ref{sect:rigidity}, we study properties of exceptional sets. In Section~\ref{sect:examples}, we present a first series of examples, illustrating the general features of the theory of balanced line bundles discussed in previous sections. In Sections~\ref{sect:fano-threefolds} and \ref{subsec: primitive rank two}, we turn to Fano threefolds and determine in which cases the 
anticanonical line bundles are balanced. In Section~\ref{sect:applications}, 
we work over number fields and discuss several arithmetic applications of 
the theory of balanced line bundles in this context.

\

\noindent
{\bf Acknowledgments.} 
The authors would like to thank Brendan Hassett, Damiano Testa, and Anthony V\'arilly-Alvarado for useful suggestions and Mihai Fulger for providing an argument for Lemma \ref{lemm:chowboundedness}.
The third author was partially supported by NSF grant 1160859.

\section{Preliminaries}
\label{sect:preliminaries}

In Sections~\ref{sect:preliminaries}-\ref{subsec: primitive rank two} 
we work over an algebraically closed field of characteristic zero.  
A variety is an irreducible reduced scheme of finite type over this field.

\subsection{Basic definitions}
\label{sect:basic}

Let $X$ be a smooth projective variety and $\NS(X)$ its N\'eron-Severi group. 
We denote the corresponding N\'eron-Severi space by 
$\mathrm{NS}(X, \bR)=\NS(X)\otimes \bR$ and the cone of pseudo-effective divisors, i.e., 
the closure of the cone of effective $\bR$-divisors in $\mathrm{NS}(X, \bR)$, by  
$\Lambda_{\mathrm{eff}}(X)$.   The interior of $\Lambda_{\mathrm{eff}}(X)$ is known as the big cone and is denoted by $\mathrm{Big}^{1}(X)$.
We identify divisors and line bundles with their classes in $\NS(X,\bR)$, when convenient. 
We write $K_X$ for the canonical class of $X$ 
and $\mathrm{Bs}(L)$ for the base locus of a line bundle $L$.

\begin{defi}
Let $X$ be a normal projective variety of dimension $n$ and $L$ a Cartier divisor on $X$.  
The volume of $L$ is 
\begin{equation*}
\vol(L) = \limsup_{m \to \infty} \frac{h^{0}(X,mL)}{m^{n}/n!}.
\end{equation*}
\end{defi}
We record the following facts about the volume (see \cite{LazI}):
\begin{itemize}
\item $\vol$ is homogeneous: $\vol(aL) = a^{n}\vol(L)$.  Thus $\vol$ can naturally be defined for $\mathbb{Q}$-divisors as well.
\item $\vol$ extends to a continuous function $\NS(X,\bR) \to \mathbb{R}$.
\item $\vol(L) > 0$ precisely when $L$ lies in the interior of $\Lambda_{\rm eff}(X)$.
\end{itemize}

\begin{defi}
Let $X$ be a normal projective variety and $L$ a pseudo-effective $\mathbb{Q}$-Cartier divisor on $X$.  The stable base locus of $L$ is
\begin{equation*}
\mathbf{B}(L) := \bigcap_{m \in \mathbb{Z}_{>0}, mL\textrm{ Cartier}} \mathrm{Bs}(mL).
\end{equation*}
We will often work with the following perturbed versions.  The augmented base locus is
\begin{equation*}
\mathbf{B}_{+}(L) := \bigcap_{A \textrm{ ample }\mathbb{Q}\textrm{-Cartier}} \mathbf{B}(L-A).
\end{equation*}
\cite[Proposition 1.5]{elmnp06} verifies that there is some ample $\mathbb{Q}$-Cartier divisor $A$ such that $\mathbf{B}_{+}(L) = \mathbf{B}(L-A)$.  In particular $\mathbf{B}_{+}(L)$ is a closed subset.  When $L$ is big, then $L|_{Y}$ is again big for any subvariety $Y \not \subset \mathbf{B}_{+}(L)$.

The diminished base locus is
\begin{equation*}
\mathbf{B}_{-}(L) := \bigcup_{A \textrm{ ample }\mathbb{Q}\textrm{-Cartier}} \mathbf{B}(L+A).
\end{equation*}
The results of \cite{BCHM} show that this set is closed for $L = K_{X} + \Delta + A$ when $(X,\Delta)$ is a normal $\mathbb{Q}$-factorial klt pair and $A$ is an ample divisor.
\end{defi}

\subsection{Families of subvarieties}

We describe families of subvarieties on a smooth projective variety $X$ 
using the Chow variety of \cite{kollarbook}.  
Fix an ample divisor $L$ on $X$.  \cite[I.3.10 Definition]{kollarbook} 
defines a family of $r$-dimensional algebraic cycles on $X$ of $L$-degree $d$ over a base $W$.  
One can think of a family as a cycle $U$ on $X \times W$ which is ``essentially flat'', 
although the precise definition is more involved.  
We use the notation $\pi: U \to W$ to denote a family of this kind.

\cite{kollarbook} defines a functor $Chow_{r,d}(W)$ parametrizing 
effective families over $W$.  By \cite[I.3.21 Theorem]{kollarbook}, for seminormal schemes
this functor is representable by a seminormal projective variety $\Chow_{r,d}(X)$.  
We let $\Chow(X)$ denote the disjoint union over all $r$ and $d$ of $\Chow_{r,d}(X)$.

By a family of subvarieties of $X$, we 
mean a family of effective cycles on $X$ 
whose general member is irreducible and reduced.

\subsection{Minimal model program}
\label{sect:mmpintro}

We frequently use the following result of \cite{BCHM}, 
sometimes without making an explicit reference.

\begin{theo}[\cite{BCHM}] 
\label{theo:runningmmp}
Let $X$ be a smooth projective variety and $L$ a big and nef $\mathbb{Q}$-divisor on $X$.  
Suppose that $K_{X} + L$ is pseudo-effective.  Then there is a birational contraction 
$\phi: X \dashrightarrow X'$ to a $\mathbb{Q}$-factorial terminal variety $X'$ such that 
$K_{X'} + \phi_{*}L$ is semiample.
\end{theo}

We will (somewhat abusively) say that $X'$ is a minimal model for 
$(X,L)$ and that the image $Z$ of the semiample fibration on $X'$ is a canonical model for $(X,L)$.

\begin{proof}
By Wilson's theorem (see \cite[Theorem 2.3.9]{LazI}), 
there exists an effective divisor $E$ such that for sufficiently small $t$, $L - tE$ is ample. Choose a $t\in \bQ$ small enough so that the pair $(X, tE)$ is terminal. It follows from \cite[Theorem 4.8]{Ko97} that there exists a $\bQ$-divisor $L'$ which is $\bQ$-linearly equivalent to $L-tE$ such that $(X, L'+tE)$ is terminal. By \cite{BCHM}, after a sequence of $K_{X}+L'+tE$-flips and divisorial contractions $\phi: X \dashrightarrow X'$ we obtain a klt pair $(X',\phi_*(L'+tE))$ such that $K_{X'} + \pi_{*}L$ is semiample.  The singularities of $X'$ are terminal by the following lemma.
\end{proof}

\begin{lemm} \label{lemm:terminalsings}
Let $X$ be a smooth variety and $L$ a big and nef $\mathbb{Q}$-divisor on $X$.  Suppose that $\phi: X \dashrightarrow X'$ is a finite sequence of $K_{X}+L$-flips and divisorial contractions.  Then there is some effective $\mathbb{Q}$-divisor $L' \equiv L$ such that $(X',\phi_{*}L')$ has terminal singularities.
\end{lemm}

\begin{proof}
Fix an ample divisor $A$ on $X$.  For sufficiently small rational $\epsilon > 0$, each step of $\phi$ is also a step of the $K_{X} + L + \epsilon A$-MMP.  Since $L+\epsilon A$ is ample, it is $\mathbb{Q}$-linearly equivalent to an ample divisor $D$ such that $(X,D)$ is terminal and $D$ does not contain any divisorial exceptional component of $\phi$.  Since the steps of the MMP are numerical, $X'$ is again a sequence of $K_{X} + D$-flips and divisorial contractions.  By \cite[Corollaries 3.42 and 3.43]{KM98} the pair $(X',\phi_{*}D)$, and hence also $X'$, has terminal singularities.

Now, although it is not true in general that a divisorial contraction of a terminal pair $(\widetilde{X},\widetilde{L})$ preserves the terminal property, it is true if the target has terminal singularities and the coefficient of the contracted divisor is sufficiently small in $\widetilde{L}$.  By applying Wilson's Theorem (see \cite[Theorem 2.3.9]{LazI}) to $L$, we see that we can choose an $L' \equiv L$ so that every contracted divisor has arbitrarily small coefficient in $L'$, and our claim follows.
\end{proof}

\begin{prop} \label{prop:ratconnectedprop}
Let $X$ be a smooth projective variety and $L$ a big and nef $\mathbb{Q}$-divisor.  
Suppose that $K_{X} + a(X,L)L$ has Iitaka dimension $0$.  
Then $X$ is birational to a log-Fano variety and is rationally connected.
\end{prop}

\begin{proof}
Applying MMP as above, one finds a divisor $L'$ that is $\mathbb{Q}$-linearly equivalent to $a(X,L)L$ and a klt pair $(X',\pi_{*}L')$ with $K_{X'} + \pi_{*}L' \equiv 0$.  Note that $\pi_{*}L'$ is big, so we can write $\pi_{*}L' \equiv E + A$ for an effective $\mathbb{Q}$-divisor $E$ and an ample $\mathbb{Q}$-divisor $A$.  
Then for sufficiently small $\epsilon$, $(X',(1-\epsilon)\pi_{*}L' + \epsilon E)$ is still a klt pair.  But since
\begin{equation*}
K_{X'} + (1-\epsilon) \pi_{*}L' + \epsilon E \equiv -\epsilon A
\end{equation*}
 is anti-ample, we conclude that $X$ is birational to a log-Fano variety.  Then $X$ is rationally connected by \cite{HM07}.
\end{proof}

\subsection{Rational curves}
Rather than parametrizing rational curves as subvarieties of $X$, it will 
be more convenient to use the space $\Mor(\mathbb{P}^{1},X)$ as in \cite[I.1.9 Definition]{kollarbook}.
We use the following fundamental results concerning the deformation of rational curves.

\begin{lemm}[\cite{kollarbook}, II.5.14 Theorem]
\label{lemm:bendandbreak}
Let $X$ be a smooth projective variety of dimension $n$. 
Suppose that  $K_{X} + L$ is not pseudo-effective.  Then there is a rational curve $C$ on $X$ such that
\begin{itemize}
\item $C$ deforms to cover $X$,
\item $-K_{X} \cdot C \leq n+1$.
\end{itemize}
\end{lemm}

\begin{lemm}[\cite{kollarbook}, II.3.10.1 Corollary] \label{lemm:deformationcount}
Let $X$ be an smooth projective variety of dimension $n$.  Let $\pi: \mathcal{C} \to S$ be a family of rational curves admitting a dominant morphism $s: \mathcal{C} \to X$.  Then, for a general member $C$ of the family $\pi$ we have:
\begin{enumerate}
\item $T_{X}|_{C}$ is nef.
\item $\dim_{f|_{C}} \mathrm{Mor}(\mathbb{P}^{1},X) = -K_{X} \cdot C + n$.
\end{enumerate}
In particular, $-K_{X} \cdot C \geq 2$.
\end{lemm}

\subsection{Fujita-type statements}

Other tools in the study of adjoint divisors are vanishing theorems.  
These tend to lead to better results than the corresponding versions for rational curves, 
but require stronger hypotheses.  
We use the following Fujita-type statements:

\begin{theo}[\cite{Reider}]
\label{theo: Reider}
Let $X$ be a smooth projective surface and $L$ a nef divisor on $X$ with $L^{2} \geq 5$.  
If $|K_{X}+L|$ has a basepoint at $x \in X$, then there is an effective divisor $D$ containing $x$ satisfying
\begin{align*}
L \cdot D = 0 \textrm{  and  } & D^{2} = -1  \textrm{  or}\\ 
L \cdot D = 1 \textrm{  and  } &  D^{2} = 0.
\end{align*}
\end{theo}

\begin{theo}[\cite{Reider}]
\label{theo: Reider II}
Let $X$ be a smooth projective surface and $L$ a nef divisor on $X$ with $L^2 \geq 10$. 
If $|K_X+L|$ fails to separate two points $x, y$ (possibly infinitely near), 
then there is an effective divisor $D$ through $x$ and $y$ such that
\begin{align*}
L \cdot D = 0 \textrm{ and } & D^2 = -1 \textrm{ or } -2; \textrm{ or }\\
L \cdot D = 1 \textrm{ and } & D^2 = 0 \textrm{ or } -1; \text{ or }\\
L \cdot D = 2 \textrm{ and } & D^2 = 0.
\end{align*}
\end{theo}

In general, we have:

\begin{prop} \label{prop:siubound}
Let $X$ be a smooth projective variety of dimension $n$ and 
$L$ a big and nef divisor on $X$.  Then $K_{X} + (n+1)L$ is pseudo-effective.
\end{prop}

\begin{proof}
The argument is due to Siu.  
Put $P(m) := \chi(\mathcal{O}_{X}(K_{X} + mL))$.  By Hirzebruch-Riemann-Roch, 
$P(m)$ is a polynomial in $m$ of degree at most $n$.  By Kawamata-Viehweg vanishing, for $m>0$ we have
\begin{equation*}
P(m) = H^{0}(X,K_{X} + mL).
\end{equation*}
Note that $P(m)$ is not identically zero since for $m$ sufficiently large  we have 
$H^{0}(X,K_{X}+mL) > 0$.  Thus $P(m)$ cannot have $n+1$ roots, so that $H^{0}(X,K_{X}+mL) > 0$ for some $1 \leq m \leq n+1$.  
Increasing the coefficient of $L$, we see that $K_{X} + (n+1)L$ must be pseudo-effective.
\end{proof}

However, as in Reider's result it is useful to have an intersection-theoretic criterion instead.  The main theorem in this direction is due to \cite{AS95}.

\begin{theo}[\cite{Ko97}, Theorem 5.8]
Let $X$ be a smooth projective variety of dimension $n$ and $L$ a big and nef Cartier divisor on $X$.  Fix a point $x \in X$ and assume that
\begin{equation*}
L^{\dim Z} \cdot Z > \left( \begin{array}{c} n+1 \\ 2 \end{array} \right)^{\dim Z},
\end{equation*}
for every irreducible subvariety $Z$ passing through $x$.  Then $K_{X} + L$ has a section that does not vanish at $x$.
\end{theo}
Conjecturally, one can replace the right side by $n^{\dim Z}$.

\section{Balanced line bundles}
\label{sect:balanced}

Here we study the invariants appearing in Manin's conjecture \eqref{eqn:manin}.

\subsection{$a$-constants}

\begin{defi}\cite[Definition 2.2]{balanced}
\label{defi: Fujita invariant}
Let $X$ be a smooth projective variety 
and $L$ a big Cartier divisor on $X$. 
The {\it Fujita invariant} is
$$
a(X, L) := \min \{ t\in \bR \mid t[L] + [K_X] \in \Lambda_{\mathrm{eff}}(X) \}.
$$
\end{defi}

By \cite[Proposition 7]{balanced}, $a(X, L)$ is a birational invariant. 
Hence, we define the Fujita invariant for a singular projective variety $X$ 
by taking a smooth resolution $\beta : \tilde{X} \ra X$:
$$
a(X, L):= a(\tilde{X}, \beta^*L).
$$
This definition does not depend on the choice of $\beta$.  
By \cite{BDPP}, $a(X,L)$ is positive if and only if $X$ is uniruled.  


\begin{lemm}
Let $X$ be a smooth projective variety.  The function 
$$
a(X,-): \mathrm{Big}^{1}(X) \to \mathbb{R}
$$ 
is continuous.
\end{lemm}

\begin{proof}
Fix a big divisor $L$ and consider the closure $\mathcal{T}$ of a small neighborhood of the numerical class 
$L$ in the big cone.  Note that the $a$ values for $D\in \mathcal{T}$ are bounded in absolute value by some constant $M$, since $\mathcal{T} \pm \frac{1}{M}K_{X}$ lies in the big cone for large enough $M$.  Consider the set $\mathcal{T} \times [-M,M]$.  The volume map 
$$
\vol: (D,a) \mapsto \vol(K_{X} + aD)
$$ 
is continuous, so the preimage of $0$ is closed.  Similarly, the addition map 
$$
p: (D,a) \mapsto K_{X} + aD
$$ 
is continuous, so the preimage of the the pseudo-effective cone is closed.  
The intersection of these two sets is exactly the set of pairs $(D,a(X,D))$.  
Using the limit-definition of compactness, we see that $a(X,D)$  must vary continuously.
\end{proof}

\subsection{$b$-constants}

\begin{defi}
Let $V$ be a finite dimensional vector space over $\bR$ and $\Lambda \subset V$ a closed convex cone. A supporting function $\sigma : V \ra \bR$ is a linear functional such that $\sigma$ is non-negative on $\Lambda$. A supported face is a face of the form
$$
F = \Lambda \cap \{ \sigma = 0\}.
$$
It is an extremal face of $\Lambda$.
\end{defi}

\begin{defi}\cite[Definition 2.8]{balanced}
\label{defi: invariant b}
Let $X$ be a smooth projective variety with non-pseudo effective canonical class. Let $L$ be a big Cartier divisor on $X$. We define $b(X, L)$ to  be 
\begin{center}
the codimension of the minimal supported face of $\Lambda_{\mathrm{eff}}(X)$ containing the numerical class $a(X, L)[L] + [K_X]$. 
\end{center}
\end{defi}

Again, this is a birational invariant (\cite[Proposition 9]{balanced}), 
and we define $b(X, L)$ for a singular variety $X$ by taking a smooth resolution $\beta : \tilde{X} \ra X$
$$
b(X, L) := b(\tilde{X}, \beta^*L).
$$
This definition does not depend on the choice of $\beta$.

\begin{lemm} 
\label{lemm:bcomputedonmm}
Let $X$ be a smooth projective variety and $L$ a big and nef divisor on $X$.  
Let $\pi: X \dashrightarrow X'$ be a minimal model for this pair.  Then $b(X,L) = b(X',\pi_{*}L)$.
\end{lemm}

\begin{proof}
For notational convenience, we rescale $L$ so that $a(X,L)=1$.  
As in Section \ref{sect:mmpintro}, $X'$ has $\mathbb{Q}$-factorial terminal singularities 
so that we may calculate $b$-constants directly on $X'$ (with no resolution needed) by \cite[Proposition 9]{balanced}.  

Let $W$ be a smooth birational model admitting birational morphisms 
$\phi: W \to X$ and $\phi': W \to X'$.  
Then 
$$
b(X,L) = b(W,\phi^{*}L)\quad \text{ and } \quad b(X',\pi_{*}L) = b(W,\phi'^{*}\pi_{*}L).
$$  
We have
\begin{equation*}
\phi^{*}L + E = \phi'^{*}\pi_{*}L
\end{equation*}
for some effective $\phi'$-exceptional divisor $E$.  Note that the minimal supported face for $K_{W} + \phi'^{*}\pi_{*}L$ contains the minimal supported face for $K_{W} + \phi^{*}L$.   Write
\begin{equation*}
K_{W} + \phi^{*}L = P + N
\end{equation*}
for the $\sigma$-decomposition of $K_{W} + \phi^{*}L$.  Then $P$ is semiample and the support of $N$ contains the union of the $\phi'$-exceptional divisors (since by Lemma \ref{lemm:terminalsings} $(X',L')$ has terminal singularities for some $L' \equiv \pi_{*}L$).  Therefore,  
any nef curve class that has vanishing intersection with $K_{W} + \phi^{*}L$ 
also has vanishing intersection with $N$, and thus also with $E$.  
We conclude that the minimal supported faces for $K_{W} + \phi^{*}L$ and $K_{W} + \phi'^{*}\pi_{*}L$ coincide.
\end{proof}


\begin{defi}
\label{defi:poly}
Let $X$ be a $\bQ$-factorial terminal projective variety
and $D$ an pseudo-effective $\bR$-divisor which is in the boundary of $\Lambda_{\eff}(X)$.  
We say $D$ is {\it locally rational polyhedral}
if either $D$ is numerically zero or there exist finitely many linear functionals
$$
\lambda_i:\NS(X,\bQ) \ra \bQ
$$
such that $\lambda_i(D)> 0$ and
$$
\Lambda_{\eff}(X) \cap \{ v : \lambda_i(v) \geq 0 \text{ for any $i$}\},
$$
is finite rational polyhedral and generated by effective $\bQ$-divisors.
\end{defi}

Lemma~\ref{lemm:bcomputedonmm} 
is naturally compatible with the following proposition. 

\begin{prop}[\cite{balanced}, Proposition 2.18] \label{prop:binterpretationforsemiample}
Let $X$ be a uniruled $\mathbb{Q}$-factorial terminal projective variety and 
$L$ a big divisor on $X$.  Suppose that $K_{X} + a(X,L)L$ is locally rationally polyhedral and semiample and that $\kappa(K_{X}+a(X,L)L) \geq 1$.  
If $\pi: X \to Z$ denotes the morphism defined by $K_{X} + a(X,L)L$, then
\begin{equation*}
b(X,L) = \rk \, \NS(X) - \rk \, \NS_{\pi}(X),
\end{equation*}
where $\NS_{\pi}(X)$ is the lattice generated by $\pi$-vertical divisors.
\end{prop}

\begin{rema}
While \cite[Proposition 2.18]{balanced} is only stated for smooth varieties, the proof works equally well for varieties with terminal singularities.  Also, the condition $\kappa(K_{X} + a(X,L)L) \geq 1$ was inadvertently omitted. \end{rema}

In general, if $\pi: X \dashrightarrow X'$ is 
a minimal model for $(X,a(X,L)L)$, then $(X',a(X,L)\pi_{*}L)$ 
satisfies the hypotheses of Proposition \ref{prop:binterpretationforsemiample} 
and thus gives a geometric interpretation of the constant $b(X,L)$.

\begin{coro} \label{coro:binterpretation}
Let $X$ be a smooth uniruled projective variety and $L$ 
a big and nef divisor on $X$.  Suppose that $\kappa(K_{X}+a(X,L)L) \geq 1$.
Let $\pi: X \dashrightarrow Z$ be the rational map to the canonical model of $(X,a(X,L)L)$.  Then
\begin{equation*}
b(X,L) = \rk \, \NS(X) - \rk \, \NS_{\pi}(X),
\end{equation*}
where $\NS_{\pi}(X)$ is the sublattice generated by all $\pi$-vertical divisors and by all divisors contracted by the $(a(X, L)L + K_X)$-MMP.
\end{coro}

\begin{proof}
Note that if $\phi: \tilde{X} \to X$ is a blow-up of a smooth center, 
then the desired formula is true for $(X,L)$ if and only if it is true for 
$(\tilde{X},\phi^{*}L)$.  By resolving, we may assume that the rational map 
$\pi$ is a morphism, factoring through a morphism to a minimal model $\psi: X \to X'$.

Once we verify that $K_{X'} + a(X,L)\psi_{*}L$ is locally rationally polyhedral, 
we can conclude by applying Lemma \ref{lemm:bcomputedonmm} and 
Proposition \ref{prop:binterpretationforsemiample}.  First, note that the locally rational polyhedral property can be pushed forward under a birational map.  So, it suffices to show that $K_{X} + a(X,L)L$ is locally polyhedral.  Wilson's theorem shows that we can write $a(X, L)L = A + \Delta$, where $A$ is ample $\bQ$-divisor and $(X, \Delta)$ is a klt pair.  Since the hyperplane $(K_X+A+\Delta)^\perp$ does not intersect $\overline{\mathrm{NE}}_{1}(X)_{K_X + \Delta \geq 0}$, we conclude the existence of the linear functionals by \cite{L11}.  Furthermore, any pseudo-effective $\mathbb{Q}$-divisor in a sufficiently small neighborhood of $K_{X} + A + \Delta$ rescales to be of the form $K_{X} + A' +\Delta$ for an ample divisor $A'$.  Then \cite{BCHM} shows the $\mathbb{Q}$-effectiveness of the corresponding ray, yielding the result.
\end{proof}

We will also need a version in the Iitaka dimension $0$ case.

\begin{lemm} \label{lemm:binterpretationiitaka0}
Let $X$ be a smooth uniruled projective variety and $L$ a big and nef divisor on $X$.  Suppose that $\kappa(K_{X} + a(X,L)L) = 0$.  Let $\phi: X \dashrightarrow X'$ be a minimal model for $(X,a(X,L)L)$.  Then
\begin{equation*}
b(X,L) = \rk \, \NS(X') = \rk \, \NS(X) - \rk \, \NS_{\phi}(X).
\end{equation*}
\end{lemm}

\begin{proof}
This follows immediately from Lemma \ref{lemm:bcomputedonmm}.
\end{proof}

\subsection{Balanced divisors}


\begin{defi}
\label{defi: balanced line bundles}
Let $X$ be a uniruled projective variety and $L$ a big Cartier divisor on $X$. Suppose that $Y \subset X$ is an irreducible 
proper subvariety of $X$. The divisor $L$ is {\it weakly balanced with respect to $Y$} if
\begin{itemize}
\item $L$ is big on $Y$;
\item $a(Y, L) \leq a(X, L)$;
\item if $a(Y, L) = a(X, L)$, then $b(Y, L) \leq b(X, L)$.
\end{itemize}
It is {\it balanced with respect to $Y$} if it is weakly balanced and one of the two inequalities is strict. 

The divisor $L$ is {\it weakly balanced} (resp. {\it balanced}) on $X$ if there exists a Zariski closed subset $Z \subsetneq X$ such that $L$ is weakly balanced (resp. balanced) with respect to every $Y$ not contained in $Z$. The subset $Z$ will be called {\it exceptional}.
\end{defi}

Note the slight incongruity that an exceptional set for a balanced pair $(X,L)$ 
might differ from an exceptional set when we consider $(X,L)$ as a weakly balanced pair. 
We also use a weaker version of the balanced property:

\begin{defi}
\label{defi: a-balanced}
Let $X$ be an uniruled projective variety and $L$ a big Cartier divisor on $X$. Let $Y \subset X$ be an irreducible 
proper subvariety of $X$. The divisor $L$ is {\it weakly $a$-balanced with respect to $Y$} if 
$L$ is big on $Y$ and $a(Y, L) \leq a(X,L)$. It is {\it $a$-balanced with respect to $Y$} if we have a strict inequality.

The divisor $L$ is {\it weakly $a$-balanced} if there exists a proper Zariski closed subset $Z \subset X$ such that $L$ is weakly $a$-balanced with respect to any $Y$ not contained in $Z$, or {\it $a$-balanced} in the case of strict inequality.  The subset $Z$ will be called $a$-exceptional.  
\end{defi}

\section{Properties of the exceptional set}
\label{sect:rigidity}

In this section we compare the value of $a(X,L)$ to $a(Y,L)$ for subvarieties $Y$.   To prove the strongest statements, we will need the following special case of the Borisov-Alexeev-Borisov Conjecture.

\begin{babconj}
There is a positive number $\delta(n)$ such that every $\mathbb{Q}$-factorial terminal Fano variety $X$ of Picard number $1$ with dimension $\leq n$ satisfies
\begin{equation*}
(-K_{X})^{\dim X} \leq \delta(n).
\end{equation*}
\end{babconj}

Conjecture WBAB$_n$ holds for $n=3$ by \cite{Kawamata92}.

\subsection{Dominant families of subvarieties}

\begin{prop}
\label{prop:verygeneralbalanced}
Let $X$ be a smooth projective variety and $L$ a big and nef $\mathbb{Q}$-divisor.  Let $\pi: \mathcal{U} \to W$ be a family of subvarieties of $X$ such that $s: \mathcal{U} \to X$ is dominant.  Then a general member $Y$ of the family $\mathcal{U}$ 
satisfies $a(Y,L) \leq a(X,L)$.
\end{prop}

\begin{proof}
Let $\mathcal{U}'$ be a resolution of the universal family and let $\pi'$ and $s'$ be the corresponding maps.  By \cite[6.8 Theorem]{Peternell12}, any smooth fiber $Y'$ of $\pi'$ with trivial normal bundle has the natural map $\Nef_{1}(Y') \rightarrow \Nef_{1}(\mathcal{U}')$ under the natural inclusion.  In particular, for any nef curve class $\alpha$ on $Y'$ we have that $(s'|_{Y'})_{*}\alpha$ is nef on $X$.  For such a $Y'$, let $Y$ denote its image in $\mathcal{U}$.

Let $d$ denote the dimension of $Y$ and $n$ the dimension of $X$.  By cutting down $W$ by general hyperplane sections to an $(n-d)$-dimensional base $T$, we obtain a smooth restricted family $\mathcal{U}'_{T}$ containing $Y'$ whose map to $X$ is generically finite.  Then
\begin{equation*}
K_{\mathcal{U}'_{T}} \equiv s'^{*}K_{X} + E
\end{equation*}
for some effective divisor $E$.  Since $Y'$ has trivial normal bundle in $\mathcal{U}'_{T}$ as well, \cite[6.8 Theorem]{Peternell12} still applies to $\mathcal{U}'_{T}$.

Consider a nef curve class $\alpha$ on $Y'$.   
Letting $i: Y' \to \mathcal{U}'_{T}$ denote the inclusion,
\begin{align*}
(K_{Y'} + a(X,L)(s'|_{Y'})^{*}L) \cdot \alpha & = (K_{\mathcal{U}'_{T}} + a(X,L)(s'|_{\mathcal{U}'_{T}})^{*}L) \cdot i_{*}\alpha \\
& \geq (K_{X} + a(X,L)L) \cdot (s'|_{Y'})_{*}\alpha.
\end{align*}
In particular, since a divisor is pseudo-effective exactly when it has non-negative intersection against every nef curve class,
$$
a(Y,L) = a(Y',s'^{*}L) \leq a(X,L).
$$
\end{proof}

\begin{theo} \label{theo:familybalanced}
Let $X$ be a smooth projective variety and $L$ a big and nef $\mathbb{Q}$-divisor.  Let $\pi: \mathcal{U} \to W$ be a family of subvarieties of $X$.  There exists a proper closed subset $V \subset X$ such that if a member $Y$ of the family $\mathcal{U}$ satisfies $a(Y,L) > a(X,L)$ then $Y \subset V$.
\end{theo}

\begin{proof}
Let $f: \mathcal{U} \to X$ be the family.  
If $f(\mathcal{U}) \subsetneq X$ then we can set $V$ to be this subset.  Otherwise, Proposition \ref{prop:verygeneralbalanced} shows that a general member of $\mathcal{U}$ satisfies 
$a(Y,L) \leq a(X,L)$.  Let $W' \subset W$ be the closure of the subset over which this inequality does not hold.  Arguing by Noetherian induction on the components of $W'$, we obtain the proof.
\end{proof}

Using the countability of the Hilbert scheme, we immediately obtain:

\begin{coro}
Let $X$ be a smooth projective variety and $L$ an ample divisor on $X$.  There is a countable union $V$ of proper closed subsets of $X$ such that any subvariety $Y$ satisfying
\begin{equation*}
a(Y,L) >  a(X,L)
\end{equation*}
is contained in $V$.
\end{coro}

\begin{exam} \label{exam:k3surface}
The countability of $V$ is necessary: consider for example a K3 surface $X$ containing infinitely many rational curves. Then 
$a(X,L)=0$, for every ample $L$, and 
$$
a(Y,L)=2/\deg(Y) >0,
$$ 
where $\deg(Y):=L.Y$ is the $L$-degree of a rational curve $Y\subset X$. 
\end{exam}

\subsection{Rigidity}

In \cite{balanced}, the balanced property was explored for del Pezzo surfaces and equivariant compactifications of homogeneous spaces.  For del Pezzos, the balanced property for $L$ is equivalent to the rigidity of the adjoint divisor 
$a(X, L)L+K_X$. This was also observed in \cite{balanced} for flag varieties and toric varieties, 
and it is compatible with the conjectural picture of the Tamagawa constant in \cite{BT}.  
We show that this observation is true in general, proving a conjecture of \cite{BT}.

\begin{theo} \label{theo:rigidity}
Let $X$ be a smooth uniruled variety and $L$ a big and nef divisor on $X$.  
If $L$ is balanced, then $K_{X} + a(X,L)L$ is rigid.  
In particular, $X$ is rationally connected and $K_{X} + a(X,L)L$ is numerically equivalent to an exceptional divisor for a birational map.
\end{theo}

The converse is false; in Section \ref{sect:fano-threefolds} we will see many examples of Fano threefolds for which $-K_{X}$ is weakly balanced but not balanced.

\begin{proof}
Suppose that $K_{X} + a(X,L)L$ is not rigid, or equivalently, the Iitaka dimension of $K_{X} + a(X,L)L$ is at least $1$.  By resolving, we may suppose that $X$ admits a morphism $\phi: X \to X'$ to a minimal model $X'$ for $(X,a(X,L)L)$.  Let $\pi: X' \to Z$ denote the canonical model.

Theorem \ref{theo:familybalanced} shows that if $Y$ is a general fiber of $\pi \circ \phi: X \to Z$ then $a(Y,L) \leq a(X,L)$.  In fact equality is achieved since the restriction of $K_{X} + a(X,L)L$ to $Y$ is not big.

We show that furthermore $b(Y,L) \geq b(X,L)$.  By Corollary~\ref{coro:binterpretation}, 
$$
b(X,L) = \rk \, \NS(X') - \rk \, \NS_{\pi}(X').
$$ 
Similarly, let $Y'$ denote the image of $Y$, so that $Y'$ is a fiber of $\pi$.  We have $(K_{X'}+a(Y,L)\phi_{*}L)|_{Y'} \equiv 0$.  By \cite[Lemma 3.38]{KM98}, we can write
\begin{equation*}
K_{X} + a(X,L)L = \phi^{*}(K_{X'} + a(X,L)\phi_{*}L) + E
\end{equation*}
for some effective $\phi$-exceptional divisor $E$.  By the generality of $Y$, the restriction $E|_{Y}$ is also $\pi|_{Y}$-exceptional.  We see that $K_{Y} + a(X,L)L|_{Y}$ is a $\pi|_{Y}$-exceptional divisor.  Then Lemma \ref{lemm:binterpretationiitaka0} shows that
\begin{equation*}
b(Y,L) \geq \rk \, \NS(Y) - \rk \, \NS_{\pi}(Y) \geq \rk \, \NS(Y').
\end{equation*}

To conclude, it suffices to show that the restriction map $\NS(X') \to \NS(Y')$ has kernel $\NS_{\pi}(X')$.  This is true because the fibers of $\pi$ are rationally connected: suppose that $L$ is a $\bQ$-divisor satisfying $L|_{Y'} \equiv 0$.  Choose a sufficiently large integer $m$ so that $mL$ is Cartier.  Let $\psi: \widehat{X} \to X'$ be a resolution and $\widehat{Y}$ denote the strict transform of $Y'$.  By generality of $Y'$, we may suppose that $\widehat{Y}$ is smooth.   Let $\widehat{Y}_{\nu}$ denote the base change of the generic fiber of $\pi \circ \phi$ to the algebraic closure of the function field.  Then \cite[Proposition 3.6]{poonen} constructs a specialization map $\NS(\widehat{Y}_{\nu}) \to \NS(\widehat{Y}_{t})$, compatible with pulling back from $\widehat{X}$, which is injective for every smooth fiber $\widehat{Y}_{t}$ of $\pi \circ \psi$.  In particular, since $m\psi^{*}L|_{\widehat{Y}} \equiv 0$, we also have that $m\psi^{*}L|_{\widehat{Y}_{t}} \equiv 0$ for a general fiber of $\pi \circ \psi$.

Since a general fiber $Y'_{t}$ of $\pi$ is rationally chain connected with terminal singularities, using \cite{HM07} and passing to a resolution we see that for a general fiber $\widehat{Y}_{t}$ of $\pi \circ \psi$ we have $m\psi^{*}L|_{\widehat{Y}_{t}} \sim 0$.  Then using Grauert's theorem over an open subset of $Z$, we determine that $\psi^{*}L$ is $\mathbb{Q}$-linearly equivalent to a pullback of a divisor from $Z$ plus some $\pi$-vertical divisors.  By pushing forward we determine that $L$ has the same property.
\end{proof}

\subsection{Closedness of the $a$-exceptional set}

The following proposition describes the basic dichotomy governing the strongly $a$-balanced property.

\begin{prop} \label{prop:adichotomy}
Let $X$ be a smooth uniruled projective variety of dimension $n$ and $L$
a big and nef $\mathbb{Q}$-divisor.  Then either
\begin{enumerate}
\item $X$ is covered by varieties $Y$ satisfying $a(Y,L) = a(X,L)$, or
\item $X$ is birational to a $\mathbb{Q}$-factorial terminal Fano variety $X'$ of Picard number $1$.
\end{enumerate}
\end{prop}

\begin{proof}
Let $\phi: X \dashrightarrow X'$ be the result of the $K_{X} + a(X,L)L$-MMP.  We split the analysis into two cases.

First, suppose that the Iitaka dimension of $K_{X} + a(X,L)L$ is at least $1$.  Then the fibers of the Iitaka fibration satisfy property $1$ as argued in the proof of Theorem \ref{theo:rigidity}.

Next, suppose that $K_{X'} + a(X,L)\phi_{*}L$ is numerically trivial.  Choose a sufficiently small $\epsilon>0$ and continue to run the $K_{X} + (1-\epsilon)a(X,L)L$-MMP.  The result will be a birational map $\psi: X \dashrightarrow \widetilde{X}$ (which WLOG we may assume is a morphism) and a Mori fibration $\pi: \widetilde{X} \to Z$ such that $K_{\widetilde{X}} + a(X,L)\psi_{*}L$ is numerically trivial along the fibers of $\pi$.  If $\dim Z>0$, let $\widetilde{Y}$ be a general fiber of $\pi$ and let $Y$ be its preimage on $X$.  A general complete intersection curve $\widetilde{C}$ in $\widetilde{Y}$ will avoid every $\psi$-exceptional center.  Since $\psi$ is a birational contraction, the strict transform $C$ on $Y$ is nef and avoids the $\psi$-exceptional locus so that
\begin{equation*}
(K_{Y} + a(X,L)L) \cdot C = 0.
\end{equation*}
By Theorem \ref{theo:familybalanced}, we have $a(Y,L) = a(X,L)$.  Otherwise, $\dim Z = 0$ and $\widetilde{X}$ is a $\mathbb{Q}$-factorial terminal Fano variety of Picard number $1$.
\end{proof}

We will also need a lemma concerning big and nef divisors.  It was used in earlier joint work of the first author with Mihai Fulger.

\begin{lemm} \label{lemm:chowboundedness}
Let $X$ be a smooth projective variety and $L$ a big and nef $\mathbb{Q}$-divisor on $X$.  Fix a constant $C$.  Then the subset of $\Chow(X)$ parametrizing subvarieties of $X$ that are not contained in $\mathbf{B}_{+}(L)$ and are of $L$-degree at most $C$ is bounded.
\end{lemm}

\begin{proof}
By definition, there is an effective $\mathbb{Q}$-divisor $E$ such that $L \equiv A + E$ for an ample $\mathbb{Q}$-divisor $A$ and the stable base locus of $E$ is equal to $\mathbf{B}_{+}(L)$.  Suppose that $Z$ is a $d$-dimensional subvariety not contained in $\mathbf{B}_{+}(L)$.  Then there is a divisor $E'$ that is $\mathbb{Q}$-linearly equivalent to $E$ whose support does not contain $Z$.  Note that
\begin{equation*}
L^d-(L-E)^d \equiv E' \cdot \left( \sum_{i=1}^d  L^{i-1}(L-E)^{d-i} \right)
\end{equation*}
where the term in parentheses is a positive combination of complete intersections of nef divisors.  Since $E' \cdot Z$ is an effective cycle, we have that
\begin{equation*}
L^{d} \cdot Z \geq (L-E)^{d} \cdot Z
\end{equation*}
and we conclude by the usual boundedness for ample divisors.
\end{proof}

\begin{theo} \label{theo:curveclosedness}
Let $X$ be a smooth uniruled projective variety and $L$ 
a big and nef $\mathbb{Q}$-divisor on $X$.  There is a proper closed $V \subset X$ such that any integral curve $C$ with $a(C,L) > a(X,L)$ is contained in $V$.
\end{theo}

\begin{proof}
Since $a(X,L)>0$ it suffices to consider rational curves $C$.  Suppose that $C \not \subset \mathbf{B}_{+}(L)$.  Then $L \cdot C > 0$ and $a(C,L) = \frac{2}{L \cdot C}$.  Thus, if $a(C,L) > a(X,L)$ then $L \cdot C < \frac{2}{a(X,L)}$.  By Lemma \ref{lemm:chowboundedness}, there are only finitely many families of such curves, and we conclude by Theorem \ref{theo:familybalanced}.
\end{proof}

In order to extend this result to higher dimension subvarieties, we will need to apply a boundedness conjecture for terminal $\mathbb{Q}$-Fanos.

\begin{lemm} \label{lemm:babimpliesboundedness}
Assume Conjecture WBAB$_{n}$.  Let $X$ be a projective variety of dimension $n$ and $L$ a big and nef $\mathbb{Q}$-divisor on $X$.  Then either:
\begin{enumerate}
\item $X$ is dominated by subvarieties $Y$ such that $a(Y,L) = a(X,L)$, or
\item $L^{n} \leq \delta(n)/a(X,L)^{n}$.
\end{enumerate}
\end{lemm}

\begin{proof}
By Proposition \ref{prop:adichotomy}, we may assume there is a birational contraction $\phi: X \dashrightarrow X'$ to a $\mathbb{Q}$-factorial terminal Fano variety $X'$ of Picard number $1$ and that $K_{X'} + a(X,L)\phi_{*}L$ is numerically trivial.  By Conjecture WBAB$_{n}$ we have
\begin{equation*}
a(X,L)^{n}\vol(\phi_{*}L) \leq \delta(n).
\end{equation*}
Note that 
$$
\vol(L) \leq \vol(\phi^{*}\phi_{*}L) = \vol(\phi_{*}L),
$$ 
since $\widetilde{X}$ is normal and $\mathbb{Q}$-factorial.  Since $\vol(L) = L^{n}$, we obtain the desired statement.
\end{proof}

\begin{theo} \label{theo:aclosedness}
Assume Conjecture WBAB$_{n-1}$.  Let $X$ be a smooth uniruled projective variety of dimension $n$ and $L$ a big and nef $\mathbb{Q}$-divisor on $X$.  There exists a proper closed subset $V \subset X$ such that every subvariety $Y$ with $a(Y,L) > a(X,L)$ is contained in $V$.
\end{theo}

\begin{proof}
We construct $V$ inductively by increasing induction on the dimension of the $Y$.  Theorem \ref{theo:curveclosedness} shows that there is a proper closed subset $V_{1}$ so that any curve $C$ with $a(C,L) > a(X,L)$ is contained in $V_{1}$.

Suppose we have constructed a proper closed subset $V_{i}$ such that any subvariety $Y$ of dimension at most $i$ satisfying $a(Y,L) > a(X,L)$ is contained in $V_{i}$.  We construct $V_{i+1}$ as follows.  Suppose that $Y$ is an $(i+1)$-dimensional subvariety satisfying $a(Y,L) > a(X,L)$.  Lemma \ref{lemm:babimpliesboundedness} shows that either:
\begin{itemize}
\item $Y$ is covered by proper subvarieties $Z$ with 
$$
a(Z,L) = a(Y,L) > a(X,L),
$$ 
or
\item $Y \cdot L^{i+1} \leq \delta(i+1)/a(Y,L)^{i+1} < \delta(i+1)/a(X,L)^{i+1}$.
\end{itemize}
In the first case, $Y$ is contained in $V_{i}$.  In the second case, $Y$ has $L$-degree bounded above by some constant.  By Lemma \ref{lemm:chowboundedness} there is a closed subset $S$ of $\Chow(X)$ parametrizing all subvarieties of this bounded degree not contained in $\mathbf{B}_{+}(L)$.  Theorem \ref{theo:familybalanced} shows that there is a proper closed subset $V'_{i+1} \subset X$ so that any subvariety $Y$ in this family with $a(Y,L) > a(X,L)$ is contained in $V'_{i+1}$.  Set $V_{i+1} = V_{i} \cup V'_{i+1} \cup \mathbf{B}_{+}(L)$.
\end{proof}

\begin{exam}
The argument above does not show that the subvarieties 
$Y$ with $a(Y,L) > a(X,L)$ have bounded $L$-degree.  In fact this stronger boundedness is not true.

For example, suppose $T \subset W$ is a smooth surface in a smooth fourfold and $A$ is an ample divisor on $W$.  Set $\pi: X \to W$ to be the blow-up along $T$ with exceptional divisor $E$ and define the ample divisor $L = \pi^{*}A - \epsilon E$ for some sufficiently small rational $\epsilon > 0$.  Note that if $C$ is a fiber of $\pi|_{E}$, then
\begin{equation*}
a(C,L) = \frac{1}{\epsilon} a(C,-E|_{C}) = \frac{2}{\epsilon}
\end{equation*}
Now let $Y$ be any surface in $E$ that is ruled by the fibers of $\pi$.  Since $a(Y,L) \geq a(C,L)$, if we choose $\epsilon$ sufficiently small then $a(Y,L) > a(X,L)$ for any such $Y$.
\end{exam}

\subsection{$b$-exceptional sets}

The $b$ constants seem more difficult to work with geometrically.  The most basic question is:

\begin{ques}
Let $X$ be a smooth uniruled projective variety and $L$ a big and nef $\mathbb{Q}$-divisor on $X$.  Consider the pairs
\begin{equation*}
(a(Y,L),b(Y,L))
\end{equation*}
as $Y$ varies over all subvarieties which deform to cover $X$.  Does this set attain a maximum (in the lexicographic order)?
\end{ques}

If so, Theorem \ref{theo:rigidity} shows that the maximum is attained on a $Y$ whose smooth model $Y'$ has the property that $K_{Y'} + a(X,L)L$ is rigid.  \cite[Section 2]{BT} proposes a stronger statement: the moving subvarieties with maximal constants define a fibration structure on $X$.  The following question reinterprets this expectation in terms of the minimal model program.

\begin{defi}
Let $X$ be a projective variety with terminal singularities and $L$ a $\mathbb{Q}$-divisor on $X$.  We say that a morphism $\pi: X \to Z$ is $L$-negative if $L$ is $\pi$-anti-ample, and $L$-trivial if $L$ is $\pi$-numerically trivial.  Similarly, if $\phi: X \dashrightarrow X'$ is a flip, we say that $\phi$ is $L$-negative or $L$-trivial based on the intersection of $L$ with the defining ray.

Now suppose that $X$ is a smooth projective variety and $L$ is a big and nef $\mathbb{Q}$-divisor.  An $a$-fibration on $X$ consists of:
\begin{itemize}
\item a finite sequence of steps $\{ \phi_{i}: X_{i} \dashrightarrow X_{i+1} \}_{i=0}^{r-1}$ of the $K_{X}$-MMP (whose composition we denote by $\psi_{i}: X \dashrightarrow X_{i}$) such that each is $K_{X_{i}}+a(X,L)\psi_{i*}L$-negative or trivial, and
\item a $K_{X}$-negative fibration $\pi: X_{r} \to Z$ that is $K_{X_{r}} + a(X,L)\psi_{r*}L$-trivial.
\end{itemize}
\end{defi}

\begin{ques} \label{ques:afibrations}
Let $X$ be a smooth uniruled projective variety and let $L$ be a big and nef $\mathbb{Q}$-divisor.  Consider pairs
\begin{equation*}
(a(Y,L),b(Y,L)),
\end{equation*}
as $Y$ varies over all subvarieties which deform to cover $X$.  
Is the maximal value of $(a(Y,L),b(Y,L))$ achieved by the strict transform of a fiber of an $a$-fibration?
\end{ques}

If the answer is yes, then the finiteness of models as in \cite{BCHM} reduces the calculation to a finite number of possibilities.  However, there does not seem to be a good way to identify the best choice of an $a$-fibration, {\it a priori}.

\begin{rema}
It is not true that any moving subvariety $Y \subset X$ with $a(Y,L)  = a(X,L)$ is a fiber of an $a$-fibration.  In Section \ref{sect:fano-threefolds}, we will see many counterexamples involving rational curves on primitive Fano threefolds.  The first difficulty concerning Question \ref{ques:afibrations} is characterizing subvarieties with $a(Y,L) = a(X,L)$.
\end{rema}

\begin{rema}
The value of the $b$-constant for an $a$-fibration is closely related to the monodromy action on the base.   See \cite{monodromy} for an analysis of this relationship (e.g., Theorem 2.6).  This seems to be a practical way to 
detect jumps in the $b$-constant.
\end{rema}

In particular, this leads to:

\begin{ques}
Let $X$ be a smooth projective variety over a number field $F$ and 
$\mathcal{L} = (L,\Vert \cdot \Vert)$ an ample adelically metrized line bundle on $X$.  
Is there a Zariski open subset $X^{\circ}\subseteq X$ such that for every sufficiently large finite extension $F'$ of the ground field $F$ one has
\begin{equation*}
N(X^\circ(F'),\mathcal L, B) \sim c B^{a(X,L)} \log(B)^{b'(X,L)-1}, \quad B\ra \infty, 
\end{equation*}
for some constant $c$, where $b'(X,L)$ is the value of $b(Y,L)$, for $Y$ the strict transform of a general fiber of some $a$-fibration?
\end{ques}

Note that even if this question is answered in the affirmative, 
the value of $c$ may be difficult to pin down.  
In particular, as discussed in \cite{BL13}, 
one may not simply take $c$ to be the Tamagawa-type constant calculated using $Y$.

\section{Examples}
\label{sect:examples}

We have seen several methods of analyzing the $a$-constant using the minimal model program, Reider-type statements, or the geometry of rational curves.  These methods indicate that subvarieties $Y$ satisfying $a(Y,L) > a(X,L)$ tend to be:
\begin{itemize}
\item covered by low degree rational curves, or
\item singular varieties of low $L$-degree, or
\item contained in $\mathbf{B}_{-}(K_{X} + a(X,L)L)$.
\end{itemize}

We first consider several general examples.

\begin{exam} \label{exam:projspacebalanced}
Each projective space $\mathbb{P}^{n}$ is balanced with respect to the hyperplane class $L$ with empty exceptional set.  To see this, first note that $a(\mathbb{P}^{n},L) = n+1$.  By Proposition \ref{prop:siubound}, $a(Y,L) \leq n$ for any variety $Y$ of dimension $\leq n-1$ and for any ample Cartier divisor $L$.  In particular $a(Y,L) < a(\mathbb{P}^{n},L)$ for any proper subvariety $Y$ of $\mathbb{P}^{n}$.
\end{exam}

The following conjecture gives an additional tool for analyzing the $a$-constants using the geometry of rational curves.  We will verify in Lemma \ref{lemm:rationalcurvecheck} that it holds for surfaces.

\begin{conj} \label{conj:rc}
Let $X$ be a smooth projective variety and $L$ a big and nef $\mathbb{Q}$-divisor on $X$.  If $K_{X} + L$ is not pseudo-effective, then there is a rational curve $C$ that deforms to cover $X$ satisfying $(K_{X} + L) \cdot C < 0$.
\end{conj}

\begin{exam}  Let $Q$ be a smooth quadric hypersurface in $\mathbb{P}^{n+1}$ for $n>2$. Assuming Conjecture \ref{conj:rc}, we show that $Q$ is balanced with respect to the hyperplane class $L$ with empty exceptional set.

To see this, note that $a(Q,L) = n$.  By Proposition \ref{prop:siubound}, $a(Y,L) < n$ for any subvariety $Y$ of dimension at most $n-2$.  Thus, to show the balanced property it suffices to consider divisors on $Q$.  By Lemma \ref{lemm:bendandbreak} and Conjecture \ref{conj:rc}, the only way a divisor $Y$ on $Q$ can fail to be balanced is if $Y$ is covered by lines (that is, rational curves $C$ with $L \cdot C = 1$) and the dimension of the space of lines on $Y$ is $2n-4$.  However, in this case any two points of $Y$ can be connected by a line in $Y$, showing that $Y$ must be an $(n-1)$-dimensional plane in $\mathbb{P}^{n+1}$.  Since a smooth quadric (of dimension $> 2$) cannot contain any such plane, we obtain the statement.
\end{exam}

\cite{balanced} analyzed toric varieties and equivariant compactifications of homogeneous spaces.  The main results are:

\begin{prop}[\cite{balanced}, Proposition 1.4]
Let $X$ be a smooth projective toric variety and $L$ a big divisor on $X$.  Then $L$ is balanced with respect to all subtoric varieties if and only if $K_{X} + a(X,L)L$ is rigid.
\end{prop}

\begin{prop}[\cite{balanced}, Theorem 1.3]
Let $H \subset M \subset G$ be connected linear algebraic groups.  Let $X$ be a smooth projective $G$-equivariant compactification of $H \backslash G$ and $Y \subset X$ the induced compactification of $H\backslash M$.  Assume that the projection $G \to M \backslash G$ admits a rational section.  Then $-K_{X}$ is balanced with respect to $Y$.
\end{prop}

For projective homogeneous spaces $X$, 
one expects that $X$ should have an empty exceptional set, for any big and nef divisor on $X$.   
We are able to verify this using previous results.

\begin{theo}
\label{thm:first-again}
Let $X$ be a generalized flag variety
and $L$ a big and nef $\mathbb{Q}$-divisor on $X$.  Then $X$ is weakly $a$-balanced with an empty $a$-exceptional set.
\end{theo}

\begin{proof}
Let $Y$ be any subvariety of $X$.  We can construct a dominant family of subvarieties of $X$ using the transitive group action on $Y$.  Note that the $a$ values for general members of this family agree with $a(Y,L)$.  By Proposition \ref{prop:verygeneralbalanced}, $X$ is weakly $a$-balanced with respect to $Y$.
\end{proof}

\subsection{Surfaces}
\label{sect: surfaces}

We next analyze invariants of surfaces.

\begin{prop}\cite[Lemma 4.1 and Proposition 4.2]{balanced}
\label{prop: del Pezzo}
Let $X$ be a del Pezzo surface, i.e., a smooth projective 
surface with ample anticanonical class. Let $L$ be a big divisor on $X$. Then $L$ is weakly balanced. Moreover, $L$ is balanced if and only if the $\bQ$-divisor $a(X, L)L+K_X$ is rigid.
\end{prop}

Note the equivalence between the balanced property and the rigidity of the adjoint divisor $a(X, L)L+K_X$.  Similar techniques can be used to analyze surfaces in general.  A key preliminary result is:

\begin{lemm} \label{lemm:rationalcurvecheck}
Let $X$ be a smooth surface and $L$ a big and nef $\mathbb{Q}$-divisor on $X$.  
If $K_{X} + L$ is not pseudo-effective then there is a rational curve $C$ that deforms to cover $X$ such that
\begin{equation*}
0 < L \cdot C < -K_{X} \cdot C \leq 3.
\end{equation*}
\end{lemm}

\begin{proof}
Let $\phi: X \dashrightarrow X'$ and $\pi: X' \to Z$ be the Mori fibration obtained by running the $(K_{X} + L)$-MMP.  Note that $X'$ has terminal singularities by Lemma \ref{lemm:terminalsings}, so that $X'$ is smooth.  Thus there is a rational curve on $X'$ with $-K_{X} \cdot C \leq 3$ contracted by $\pi$ and avoiding any $\phi$-exceptional center, and the strict transform of this curve satisfies the desired property.
\end{proof}

We split our analysis into two cases based on the Iitaka dimension of $K_{X} + a(X,L)L$.

\begin{prop} \label{prop:surfaceconstants}
Let $S$ be a smooth uniruled projective surface and let $L$ denote a big and nef divisor on $S$.
\begin{enumerate}
\item Suppose that $\kappa(K_{S} + a(S,L)L) = 1$.  Let $F$ be a general fiber of the Iitaka fibration.  Then
\begin{equation*}
a(S,L) = \frac{2}{L \cdot F} \qquad \qquad \textrm{and} \qquad \qquad b(S,L) = 1.
\end{equation*}
In particular $L$ is weakly balanced, but not balanced.  

\item Suppose that $\kappa(K_{S} + a(S,L)L) = 0$.  Let $\pi: S \to S'$ be a minimal model of $(S,a(S,L)L)$ and let $C$ be any nef curve on $S'$.  Then
\begin{equation*}
a(S,L) = \frac{-K_{S'} \cdot C}{\pi_{*}L \cdot C} \qquad b(S,L) = \rho(S').
\end{equation*}
In particular $L$ is balanced.  
\end{enumerate}
If $S$ is not rational, then it always lands in case (1).
\end{prop}

Note that the minimal model of $(S,a(S,L)L)$ will be smooth, by Lemma \ref{lemm:terminalsings}.

\begin{proof}
(1) Let $\pi: S \to T$ denote the fibration defined by $K_{S} + a(S,L)L$.  (Note that $\pi$ is a morphism by standard results of the MMP.)  Recall that $K_{S} + a(S,L)L$ has vanishing intersection with a rational curve that deforms to cover $S$.  In this situation the only choice is the general fiber $C$ of the ruling.  Since $K_{S} \cdot C = -2$, we obtain the desired expression for $a(S,L)$.  To determine $b(S,L)$, by Lemma \ref{lemm:bcomputedonmm} we may do the calculation on a minimal model $\phi: S \to S'$ for $K_{S} + a(S,L)L$.  On $S'$, the divisor $K_{S'} + a(S,L)\phi_{*}L$ is proportional to a general fiber of the ruling.  In particular, if $D$ has vanishing intersection against $K_{S'}+a(S,L)\phi_{*}L$ then $D$ also has vanishing intersection with any curve that is vertical for the ruling.  We conclude that $b(S,L)=1$ using the description of $N^{1}(X)$ for ruled surfaces.

We next show that $L$ is weakly balanced.  By Theorem \ref{theo:curveclosedness}, there is a finite collection of rational curves satisfying $a(C,L) > a(S,L)$.  Since $b(C,L)=1$ for each of these curves, we see that $L$ is weakly balanced.  However, $L$ is never balanced, since the general fibers $F$ of $\pi$ have the same $a$ and $b$ constants as $S$.

(2) In this case $K_{S} + a(S,L)L$ is a rigid divisor.  Let $\pi: S \to S'$ be the map to the minimal model, so that $K_{S'} + a(S,L)\pi_{*}L = 0$.  Since $\pi_{*}L$ is big and nef, we have $\pi_{*}L \cdot C \neq 0$.  Using Lemma \ref{lemm:binterpretationiitaka0} we obtain the given interpretations of the geometric constants.

We still must show that $L$ is balanced.  First suppose $\rho(S') \geq 2$.  Then it suffices to show that $S$ is weakly $a$-balanced; this follows from Theorem \ref{theo:curveclosedness}.

Otherwise, $\rho(S') = 1$ and $S' \cong \mathbb{P}^{2}$.  Note that $K_{S} + a(S,L)L$ is an exceptional divisor for the map to $\mathbb{P}^{2}$ with support equal to the entire exceptional locus.  We show that $L$ is balanced with respect to any rational curve $N$ except possibly for curves contained in the $\pi$-exceptional locus and rational curves satisfying $-K_{S} \cdot N < 2$ of bounded $L$-degree.  Recall that if $N$ is a member of a dominant family of rational curves on $S$ then $-K_{S} \cdot N \geq 2$, by Lemma \ref{lemm:deformationcount}.  Thus, the rational curves satisfying $-K_{S} \cdot N < 2$ of bounded $L$-degree form a closed subset of $S$, by Lemma \ref{lemm:chowboundedness}, and we can conclude that $L$ is balanced.

Let $N\subset S$ be a rational curve not contained in the $\pi$-exceptional locus.  Then
\begin{equation*}
a(N,L) = \frac{2}{L \cdot N} \qquad \qquad a(S,L) = \frac{3}{\deg(\pi_{*}L)}
\end{equation*}
First, suppose that $N$ does not intersect a $\pi$-exceptional curve.  Then 
$$
L \cdot N \geq \deg(\pi_{*}L),
$$ 
so that $a(S,L) > a(N,L)$. 
Now suppose that $N$ does intersect a $\pi$-exceptional curve.  Then
\begin{equation*}
a(S,L)L \cdot N > (-K_{S}) \cdot N.
\end{equation*}
If $-K_{S} \cdot N \geq 2$ then $a(N,L) < a(S,L)$ and $L$ is balanced with respect to $N$. Also, if $L \cdot N > \frac{2}{3} \deg(\pi_{*}L)$ then $L$ is balanced with respect to $N$, and we conclude that $L$ is balanced, 
by the argument above.
\end{proof}

\section{Primitive Fano threefolds with $\rho(X) = 1$}
\label{sect:fano-threefolds}

In this section and the next we investigate the geometry of smooth Fano threefolds, i.e., smooth projective threefolds with ample anticanonical class. There are $105$ deformation types; 
a list of these can be found in  \cite{mori-mukai}, \cite[Chapter 12]{iskov}, 
and in \cite{mori-mukai-corr}, which includes a previously missing case.

\subsection{Classification}
\label{sect:class}

The main invariants of Fano threefolds are:

\begin{itemize}
\item the rank of the Picard group, $\rho(X)=b(X,-K_X)$,
\item the index $r=r(X)$: 
the largest $r\in \mathbb N$ such that $K_X$ is divisible by $r$
in $\Pic(X)$,
\item the degree $d(X):=(-K_X)^3$, 
\item the Mori invariant $m=m(X)$: the smallest $m\in \mathbb N$ such that 
every point of $X$ lies on a rational curve $C$ with $-K_X\cdot C\le m$. 
\end{itemize}

The Picard rank $\rho(X)$ coincides with the second Betti number of $X$. 
There are $17$ deformation types of 
Fano threefolds with $\rho(X)=1$, classified in \cite{FanoI}, \cite{FanoII}, and \cite{iskov79}. 
There are $88$ deformation types of Fano threefolds with $\rho(X) \geq 2$,
classified by Mori-Mukai in \cite{mori-mukai} and \cite{mori-mukai-corr}.

A Fano threefold is called 
imprimitive if it is isomorphic to the blow-up of a Fano threefold along a smooth irreducible curve. 
A Fano threefold is primitive if it is not imprimitive. 
Fano threefolds of Picard rank one are primitive.

\begin{theo}
\label{thm:pic-rank1}
Let $X$ be a Fano threefold, with $\Pic(X)=\bZ L$ and 
$-K_X = r(X)L$. Then $X$ is one of the following:
\begin{itemize}
\item $\bP^3$, with $r(X)=4$, or a quadric, with $r(X)=3$, or 
\item $r(X)=2$ and $d(X)\in \{ 8, 16, 24, 32, 40\}$, or 
\item $r(X)=1$ and $d(X)\in \{ 2,4,6,8,10,12,14,16,18,22\}$. 
\end{itemize}
\end{theo}

Primitive Fano threefolds with $\rho(X) \geq 2$ are classified in \cite{mori-mukai}, 
and a detailed analysis is given in \cite{ott}. 
There are 13 deformation types, and they have the following restrictions on their structures:

\begin{theo}\cite[Theorem 1.6]{mori-mukai2}
\label{theo:primitive}
Let $X$ be a primitive Fano threefold. Then 
\begin{enumerate}
\item $\rho(X) \leq 3$,
\item if $\rho(X) = 2$, then $X$ is a conic bundle over $\bP^2$,
\item if $\rho(X) = 3$, then $X$ is a conic bundle over $\bP^1 \times \bP^1$.
\end{enumerate}
\end{theo}

Another useful description is given in \cite[Section 2]{manin}:

\begin{prop}
\label{prop:classify}
Every smooth Fano threefold over an algebraically closed field of characteristic zero
is isomorphic to one of the following:
\begin{enumerate}
\item[(1)] a generalized flag variety $P\backslash G$;
\item[(2)] a variety $X$ with $m(X)=2$;
\item[(3)] a blowup of varieties of type (1) or (2);
\item[(4)] a direct product of $\bP^1$ and a del Pezzo surface.
\end{enumerate}
\end{prop}

In the following sections we determine which smooth 
Fano threefolds $X$ have balanced $-K_X$. 

\subsection{$\rho(X)=1$ and $r(X)=3,4$}
\label{subsec:picard-rank-one}

\

When $r(X)=4$, $X=\bP^3$; by Example~\ref{exam:projspacebalanced},
$-K_X$ is balanced and the exceptional set is empty. 
When $r(X)=3$, $X\subset \bP^4$ is a quadric.

\begin{prop}
\label{prop: quadric}
Let $X\subset \bP^4$ be a smooth quadric. 
Then $-K_X$ is balanced and the exceptional set is empty.
\end{prop}

\begin{proof}
The class $-K_X$ is balanced with 
respect to curves, as $r(X) = 3$. 
Let $S \subset X$ be an irreducible surface. Choose a resolution of singularities 
$\beta : \tilde{S} \ra S$. 
The fundamental divisor $L$ is the restriction of $\mathcal O(1)$ on $\bP^4$ and $a(X, L)=3$. 
Consider the adjoint divisor $2\beta^*L|_{\tilde{S}} + K_{\tilde{S}}$. 
Since $X$ is a smooth quadric, we have 
$(2L)^2.S \geq 4L^3 = 8$. 
By Reider's theorem (Theorem~\ref{theo: Reider}), 
$2\beta^*L|_{\tilde{S}} + K_{\tilde{S}}$ is effective. Thus $a(S, L|_S) \leq 2$.
\end{proof}

\subsection{$\rho(X)=1$ and $r(X) = 2$}

\

We have $\Pic(X)=\bZ L$, with  $-K_X = 2L$. An irreducible curve $C$ with 
$-K_X\cdot C = 2$ is called a {\it  $-K_X$-conic}. 
The Hilbert scheme $F(X)$ of $-K_X$-conics is 
called {\it the Fano surface of $-K_X$-conics}.

Suppose that $d(X) \geq 24 = 8\cdot 3$. 
Then $L$ is very ample \cite[Theorem 3.2.4]{iskov} and the dimension of its linear series is given by $L^3+1$. The Fano surface $F(X)$ is smooth, projective, and of pure dimension $2$ (\cite{iskov79}). Let $\mathcal U$ be the universal family of $-K_X$-conics. The evaluation map $\mathcal U \ra X$ is surjective, hence $-K_X$-conics cover $X$. We have

\begin{prop}
\label{prop: index two degree higher}
Suppose that $d(X) \geq 8\cdot 3$. Then $-K_X$ is weakly balanced and the exceptional set is empty. 
However, $-K_X$ is not balanced.
\end{prop}

\begin{proof}
Let $L$ be the fundamental divisor of $X$. The assumption in our statement implies that $L^3 \geq 3$. The divisor $L$ is very ample and its linear series determines an embedding into $\bP^{L^3+1}$. Since the index is $2$, $a(X, L)=2$, and $-K_X$ is weakly balanced with respect to any curve. However, $-K_X$ is not balanced with respect to $-K_X$-conics, and these sweep out $X$. 
Thus $-K_X$ is not balanced.

Suppose that $S \subset X$ is an irreducible proper surface. 
Let $\beta : \tilde{S} \ra S$ be a resolution of singularities. 
Then $(2\beta^*L)^2 \geq 12$ and Theorem~\ref{theo: Reider} implies 
that $D=2\beta^*L + K_{\tilde{S}}$ is effective, thus 
$$
a(S, L) \leq a(X, L) = 2.
$$ 
Suppose that $a(S, L) = a(X, L) = 2$. We claim that the Iitaka dimension of $2\beta^*L + K_{\tilde{S}}$ is $1$. If not, then the Iitaka dimension of $D$ is $0$. In particular, two general points $x, y$ cannot be separated by $|D|$, hence Theorem~\ref{theo: Reider II} indicates that there exists a $-K_X$-conic $C$ containing $x$ and $y$. This implies that $S \subset \bP^{L^3+1}$ is a plane, but $X$ cannot contain a plane, contradiction.  By Proposition \ref{prop:surfaceconstants}, we have $b(S,L)=1$.  Thus $L$ is weakly balanced with respect to $S$.
\end{proof}

Suppose that $d(X) = 8\cdot 2 =16$. Then $X$ is a double cover of $\bP^3$ ramified along a smooth quartic surface $W \subset \bP^3$, i.e.,  $X$ is defined by
$$
f_4(x_0,x_1,x_2,x_3) = w^2,
$$
in the weighted projective space $\bP(1,1,1,1,2)$, where $\deg (x_i) = 1$, $\deg (w) =2$, and $f_4$ is a homogeneous polynomial of degree $4$. Let $L$ be the fundamental divisor. The linear system of $L$ defines the double cover $\pi : X \rightarrow \bP^3$. The Fano surface of $-K_X$-conics is two dimensional and $-K_{X}$ conics dominate $X$. (See the Remark of Proposition 1.3, Chapter III in \cite{iskov79}.) 
We need to classify all possible singularities of divisors of $|L|$. 
\begin{lemm}
\label{lemm: isolated del Pezzo of degree two}
For any $S \in |L|$, $S$ has only isolated singularities, hence $S$ is normal.
\end{lemm}
\begin{proof}
Suppose otherwise. Without loss of generality we assume that $S$ is defined by $x_0=0$. By our assumption, $S$ is singular along a curve $C$. On $C$, we have
$$
\left(\frac{\partial f_4}{\partial x_1}, \frac{\partial f_4}{\partial x_2}, \frac{\partial f_4}{\partial x_3}, -2w\right) = (0,0,0,0).
$$
Thus $C$ is contained in $P \cong \bP^2$ defined by $x_0 = w = 0$. On the other hand, $\frac{\partial f_4}{\partial x_0} =0$ defines a $1$ dimensional scheme $G$ on $P$, and $G$ meets with $C$. At $G\cap C$, $X$ is singular, contradiction. 
The last assertion follows from Serre's Criterion.
\end{proof}

\begin{lemm}
\label{lemm: singular del Pezzo of degree two}
For any $S \in |L|$, $S$ has only canonical singularities or elliptic surface singularities.
Moreover, when $S$ has an elliptic surface singularity, $S$ is isomorphic to the cone defined by
$$
g_4(x_1,x_2) =w^2,
$$
in the weighted projective space $\bP(1,1,1,2)$, where $g_4$ is a homogeneous polynomial of degree $4$ with distinct roots.
\end{lemm}
\begin{proof}
Assume that $S$ has an isolated singularity $P$. We may assume that $S$ is defined by $x_0=0$ and $P$ is given by $(0:0:0:1:0)$. We consider an affine patch $\{x_3 \neq 0\}$, then $S$ is given by
$$
q_4(x_1, x_2) = w^2,
$$
in $\mathbb{A}^3$, where $q_4$ is a polynomial of degree $4$, and $P$ is given by $x_1=0, x_2=0, w=0$. 
We apply the discussion of \cite[Section 4.25]{KM98} to this surface and conclude that $S$ has an (at worst) canonical singularity at $P$ if and only if $\mathrm{mult}_0 (q_4) \leq 3$. When $\mathrm{mult}_0 (q_4) =4$, $P$ is an elliptic surface singularity (\cite[Theorem 4.57]{KM98}).  By comparing against the local-analytic form of the equation in \cite{KM98}, we see that if $S$ has an elliptic singularity then it is a cone. 
\end{proof}
\begin{prop}
\label{prop: one two 16}
Let $X$ be a Fano threefold of $\rho(X)=1$, $r(X)=2$, and $d(X)=16$. Then  $-K_X$ is weakly balanced and the exceptional set is the empty set. However, $-K_X$ is not balanced.
\end{prop}
\begin{proof}
The anticanonical class $-K_X$ is weakly balanced with respect to any curve, however, $-K_X$ is not balanced with respect to $-K_X$-conics. Let $S \subset X$ be an irreducible surface and
$\beta : \tilde{S} \ra S$ a resolution of singularities of $S$. 
Since $(2\beta^*L) \geq 8$, Theorem~\ref{theo: Reider} implies  
that $2\beta^*L + K_{\tilde{S}}$ is effective, thus $a(S, L) \leq a(X, L) = 2$.

To show the weakly balanced property, we must consider the case when $a(S, L) = a(X, L)=2$. If $S \in |mL|$, where $m \geq 2$, then by Theorem~\ref{theo: Reider II} we conclude that the Iitaka dimension of $2\beta^*L + K_{\tilde{S}}$ is $1$ (since otherwise two general points of $S$ could be connected by a $-K_{X}$-conic). Its canonical fibration is a ruling and one can conclude that $b(S, L) =1$. Thus $-K_X$ is weakly balanced with respect to $S$. Assume that $S \in |L|$. If $S$ is smooth or has at most canonical singularities, then by the adjunction formula, we have $2L|_S + K_S = S|_S$ is ample.  Since the $a$ constant can be computed on any model with canonical singularities by \cite[Proposition 2.7]{balanced}, $a(S, L) <2$, contradiction.  Suppose that $S$ is a cone defined by
$$
g_4(x_1,x_2) = w^2,
$$
where $g_4$ is a homogeneous polynomial degree $4$ with distinct roots. There is only one singularity which is an elliptic surface singularity $(0:0:1:0)$. We apply a weighted blow up of type $(1,1,2)$ at the cone point, and obtain a smooth resolution $\beta : \tilde{S} \ra S$. The surface $\tilde{S}$ admits a projection to an elliptic curve $E$ defined $g_4 = w^2$ in $\bP(1,1,2)$.  
Thus $S$ is not rational, so that $b(S,L) = 1$ by Proposition \ref{prop:surfaceconstants}.  Then $-K_{X}$ is weakly balanced with respect to $S$.
\end{proof}

Suppose that $X$ is a Fano threefold of $\rho(X) = 1$, $r(X) = 2$, and $d(X) = 8$. Then $X$ is a hypersurface of degree $6$ in the weighted projective space $\bP(1,1,1,2,3)$ (\cite[Theorem 3.2.5]{iskov}). The variety $X$ is defined by
$$
f_6(x_0,x_1,x_2) + f_4(x_0,x_1,x_2)y + y^3 + z^2 = 0,
$$
where $\deg (x_i) = 1$, $\deg (y) = 2$, $\deg(z) = 3$ and $f_6$, $f_4$ are homogeneous polynomials of degree $6$ and $4$ respectively. Let $L$ be the fundamental divisor. Then $H^0(X, \mathcal O(L))$ is generated by $x_0,x_1,x_2$ (\cite[Proposition 3.3]{mori75}). The Fano scheme of $-K_X$-conics is $2$-dimensional, and it may be non-reduced. Again $X$ is swept out by $-K_X$-conics. Next we classify possible singularities of divisors in the linear system $|L|$. 
\begin{lemm}
\label{lemm: isolated}
For any $S \in |L|$, $S$ has only isolated singularities, hence $S$ is normal.
\end{lemm}
\begin{proof}
One can proceed as in the proof of Lemma~\ref{lemm: isolated del Pezzo of degree two}.
\end{proof}

\begin{lemm}
\label{lemm: singular del Pezzo of degree one}
For any $S \in |L|$, $S$ has only canonical singularities or elliptic surface singularities. Moreover, when $S$ has an elliptic singularity, $S$ is isomorphic to the cone defined by
$$
ax_1^6 + bx_1^4y + y^3 + z^2 = 0,
$$
in the weighted projective space $\bP(1,1,2,3)$ where $a, b$ satisfies $4b^3+27a^2 \neq 0$.
\end{lemm}
\begin{proof}
Suppose that $S$ has an isolated singularity $P$. Without loss of generality we may assume that $S$ is defined by $x_0=0$ and $P$ is given by $(0:0:1:y_0:0)$. Consider an affine patch $\{x_2 \neq 0\}$, then $S$ is defined by
$$
q_6(x) + q_4(x)y + y^3+z^2 = 0,
$$
in $\mathbb{A}^3$ where $q_6$ and $q_4$ are polynomials of degree $6$ and $4$ respectively, and $P$ is given by $x=0, y=y_0, z=0$. Suppose that $y_0 \neq 0$. Then $P$ is a canonical singularity. Indeed, compute the Hermitian matrix
$$
\begin{pmatrix}
\frac{\partial^2 q_6}{\partial x^2} + \frac{\partial^2 q_4}{\partial x^2}y & \frac{\partial q_4}{\partial x} & 0\\ \frac{\partial q_4}{\partial x} & 6y & 0\\0&0&2
\end{pmatrix}.
$$
At $P$ it has rank $2$ or $3$. If it has rank $3$, then the surface is locally isomorphic to $x^2+y^2+z^2 = 0$ analytically and it is an $A_1$-singularity (see \cite{KM98} for the definition). If it has rank $2$, then it follows from the Weierstrass preparation theorem that the germ of the surface at $P$ is locally analytically isomorphic to
$$
z^2 +y^2\alpha(x,y) + y\beta(x) +\gamma(x) = 0,
$$
where $\alpha, \beta, \gamma$ are analytic functions such that $\alpha(0,0) \neq 0$. Replacing $y$ by $y -\frac{\beta}{2\alpha}$, we conclude that the surface is locally analytically isomorphic to
$$
z^2 + y^2 + x^m = 0,
$$
which is an $A_m$-singularity.

Suppose that $y_0 = 0$. The point $P$ is given by $(0,0,0)$ and the surface is defined by
$$
q_6(x) + q_4(x)y + y^3+z^2 = 0.
$$
Now we apply the discussion of \cite[Section 4.25]{KM98} and obtain that $S$ has at worst canonical singularity at $P$ if and only if $\mathrm{mult}_0 (q_6) \leq 5$ or $\mathrm{mult}_0 (q_4) \leq 3$. When $\mathrm{mult}_0 (q_6) \geq 6$ and $\mathrm{mult}_0 (q_4) \geq 4$, $P$ is an elliptic surface singularity (\cite[Theorem 4.57]{KM98}). By comparing against the local-analytic form of the equation in \cite{KM98}, we see that if $S$ has an elliptic singularity then it is a cone.
\end{proof}

Now we have
\begin{prop}
\label{prop: one two eight}
Let $X$ be a Fano threefold of $\rho(X) = 1$, $r(X) = 2$ and $d(X)=8$. Then the anticanonical class $-K_X$ is weakly $a$-balanced and the $a$-exceptional set is the empty set. The anticanonical class $-K_X$ is not balanced.
\end{prop}
\begin{proof}
Let $L$ be the fundamental divisor of $X$. We have $L^3 = 1$. The anticanonical class is weakly balanced with respect to any curve because the index is $2$. However, the anticanonical class is not balanced due to the $-K_X$-conics. Let $S$ be an irreducible surface. We need to show that $a(S, L) \leq a(X, L)=2$. If $S$ is linearly equivalent to $mL$ where $m \geq 2$, then $(2L)^2\cdot S \geq 8$. By Reider's theorem (Theorem~\ref{theo: Reider}), we have $a(S, L) \leq 2$. Suppose that $S \in |L|$. If $S$ is smooth or has at worst canonical singularities, then by the adjunction formula, we have $2L|_S + K_S = S|_S$ is ample.  Since the $a$ constant can be computed on any model with canonical singularities by \cite[Proposition 2.7]{balanced}, we conclude that $a(S, L) <2$.

Suppose that $S$ has an elliptic singularity.  We calculate $a(S,L)$. By Lemma~\ref{lemm: singular del Pezzo of degree one}, $S$ is a cone defined by
$$
ax_1^6 + bx_1^4y + y^3+z^2 = 0.
$$
The only elliptic surface singularity is $(x_1:x_2:y:z) = (0:1:0:0)$ and there is no other singular point.  \cite[Theorem 4.57]{KM98} shows that if we do a weighted blow-up of type (1,2,3) at the cone point (in the ambient $\mathbb{C}^{3}$), the resulting strict transform surface $\pi: \overline{S} \to S$ has canonical singularities and maps to an elliptic curve.  Since $S$ is not rational, we can apply Proposition \ref{prop:surfaceconstants} to conclude that $a(S,L)=2$. 
\end{proof}

To establish the weakly balanced property, one needs to classify all possible singularities of divisors in the linear system $|2L|$.

\subsection{$\rho(X) = 1$ and $r(X) = 1$}

\

Let $L$ be the fundamental divisor of $X$, i.e., $-K_X = L$. 
For simplicity, we assume that $L$ is very ample. 
This excludes two types of Fano threefolds (\cite[Proposition 4.1.11]{iskov}). 
We denote the Fano scheme of $-K_X$-conics by $F(X)$. 
We also consider $-K_X$-lines, i.e., irreducible curves $C$ such that $-K_X\cdot C = 1$, and denote the 
corresponding Hilbert scheme by $\Gamma(X)$. 
The existence of $-K_X$-lines and $-K_X$-conics has been established by Shokurov 
(see \cite[Remark 4.2.8, Theorem 4.4.13, Theorem 4.5.10]{iskov}).
It is known that $\Gamma(X)$ is of pure dimension $1$ \cite[Proposition 4.2.2]{iskov}.  
Let $Z(X)$ be the surface swept out by $-K_X$-lines. 
The family of $-K_X$-conics covers $X$ (see \cite[Chapter III, Proposition 3.4]{iskov79}).

\begin{prop}
\label{prop: one one ten}
Suppose that $X$ is a Fano threefold of $\rho(X)=1$, $r(X)=1$, and $d(X) \geq10$. 
Then $-K_X$ is weakly balanced and the exceptional set is $Z(X)$, but it is not balanced.
\end{prop}

\begin{proof}
For any curve $C$ not contained in $Z(X)$, we have $-K_X\cdot C \geq 2$. 
Thus $-K_X$ is weakly balanced with respect to any curve $C$ not contained in $Z(X)$. 
However $-K_X$-conics sweep out $X$, hence $-K_X$ is not balanced. 

Suppose that $S$ is an irreducible surface not contained in $Z(X)$. Let $\beta : \tilde{S} \ra S$ be a resolution of singularities of $S$. Since $(L|_S)^2 \geq L^3 \geq 10$, it follows from Theorem~\ref{theo: Reider} that $D = \beta^*L + K_{\tilde{S}}$ is effective. We conclude that $a(S, L) \leq a(X, L) =1$. Suppose that $a(S, L) = a(X, L) =1$. We claim that $D$ has Iitaka dimension $1$. If $D$ has Iitaka dimension $0$, then $|D|$ fails to separate general points $x, y$. By Theorem~\ref{theo: Reider II}, there is a $-K_X$-conic passing through $x, y$. This implies that the dimension of the Hilbert scheme of $-K_X$-conics on $S$ is greater than $1$. \cite[Chapter III, Proposition 3.3]{iskov79} shows that $S$ is a Veronese surface in $\bP^5$, one of its projections into a lower space, a quadric surface, or a plane. However, $X$ cannot contain these surfaces. Thus we have $\kappa (D) =1$. Its canonical fibration is a ruling, and we conclude that $b(S, L)=1$. Thus $-K_X$ is weakly balanced with respect to $S$.
\end{proof}

Suppose that $d(X) = 8$ or $6$. When $d(X) = 8$, then $X$ is a complete intersection of three quadrics in $\bP^6$. When $d(X) = 6$, then $X$ is a complete intersection of a cubic and a quadric in $\bP^5$. Using Reider's theorems we have:
\begin{prop}
\label{prop: one one eight}
Let $X$ be a Fano threefold of $\rho(X) = 1$, $r(X)=1$. Suppose that $d(X) = 8$ or $6$. Then $-K_X$ is weakly $a$-balanced with the exceptional set $Z(X)$, but is not balanced.
\end{prop}
Again, to establish the weakly balanced property, one needs to classify possible singularities of divisors in the linear system $|L|$.
Suppose that $d(X) = 4$. In this case, $X$ is a quartic threefold in $\bP^4$, defined by
$$
f_4(x_0,x_1,x_2,x_3,x_4) = 0,
$$
where $f_4$ is a homogeneous polynomial of degree $4$. To study the balanced property of the anticanonical class, we need to classify all possible singularities of divisors in $|L|$. 

\begin{lemm}
\label{lemm: isolated quartic}
For any $S \in |L|$, $S$ has only isolated singularities, hence $S$ is normal.
\end{lemm}
\begin{proof}
One can discuss as Lemma~\ref{lemm: isolated del Pezzo of degree two}.
\end{proof}
A complete classification of normal quartic surfaces with irrational singularities has been given in \cite{IN04}. What we need is the following proposition:

\begin{prop}[\cite{IN04}]
\label{prop: quartic surfaces}
Suppose that $S$ is a normal quartic surface with irrational singularities. Let $\beta : \tilde{S} \ra S$ be the minimal desingularization of $S$ and $L$ the hyperplace class on $S$. Then either $\beta^*L + K_{\tilde{S}}$ is nef or $S$ is isomorphic to a cone over a smooth quartic plane curve in $\bP^3$.
\end{prop}

Using Reider's theorems, we obtain:

\begin{prop}
\label{prop: one one four}
Let $X$ be a Fano threefold of $\rho(X)=1$, $r(X) = 1$, $d(X)=4$. Suppose that $-K_X$ is very ample. Then $-K_X$ is weakly $a$-balanced with the $a$-exceptional set $Z(X)$, but not balanced.
\end{prop}

Here, we need to classify possible singularities of divisors in $|2L|$ to establish weakly balanced property.

We summarize results of this section:

\begin{theo}
\label{thm:fano-rank1}
\label{theo: picard rank one}
Suppose that $X$ is a Fano threefold of $\rho(X)=1$;
\begin{itemize}
\item if $r(X) = 4$ or $3$, then $-K_X$ is balanced;
\item if $r(X) = 2$ and $d(X) \geq 16$, then $-K_X$ is weakly balanced, but not balanced;
\item if $r(X) = 2$ and $d(X) = 8$, then $-K_X$ is weakly $a$-balanced, but not balanced;
\item if $r(X)=1$ and $d(X) \geq 10$, then $-K_X$ is weakly balanced, but not balanced;
\item if $r(X)=1$ and $d(X) = 8$ or $6$ then $-K_X$ is weakly $a$-balanced, but not balanced;
\item if $r(X)=1$, $d(X) = 4$, and $-K_X$ is very ample, then $-K_X$ is weakly $a$-balanced, but not balanced.
\end{itemize}
\end{theo}

\section{Primitive Fano threefolds with $\rho(X)=2$}
\label{subsec: primitive rank two}
\label{sect:seven}

Let $\mathrm{Z}_1(X)$ be the free group generated by irreducible curves 
on an algebraic variety $X$ and
$\mathrm{N}_1(X)_\bZ = \mathrm{Z}_1(X)/\equiv$ its quotient group of 
numerical equivalence classes. 
Put $\mathrm{N}_1(X)_\bR:=\mathrm{N}_1(X)_\bZ \otimes \bR$, it is 
dual to the N\'eron Severi space $\NS(X, \bR)$. 
A basic invariant of $X$ is the cone of effective curves 
$\mathrm{NE}_1(X)\subset \mathrm{N}_1(X)_{\bZ}$, the 
convex cone  generated by classes of effective curves. 
It is known that for any Fano threefold $X$, 
the cone of effective curves $\mathrm{NE}_1(X)$ 
is rational polyhedral and finitely generated \cite{NefM82}. 
Its edges are called extremal rays. 
Each extremal ray $R$ is generated by a rational curve. We define the length of $R$ by
$$
\mu_R = \min \{ -K_X\cdot C \mid \text{ $C$ is a rational curve such that $[C]\in R$} \}
$$
and denote a rational curve which achieves the minimum value by $l_R$. 
We call it an extremal curve of $R$. For each extremal ray $R$, there exists a morphism $f:X \ra Y$ to a normal projective variety such that $f_*\mathcal O_X = \mathcal O_Y$ and for any irreducible reduced curve $C$, $[C] \in R$ if and only if $C$ is contracted by $f$ \cite{NefM82}. 
This morphism is unique up to isomorphism, and we denote it by $\mathrm{cont}_R$. It is known that for Fano threefolds, the following sequence is exact:
$$
0 \ra \Pic (Y) \ra \Pic(X) \ra \bZ \ra 0,
$$
where the second map is the pullback of line bundles and the third map is defined by 
$D\cdot l_R$, for $D \in \Pic(X)$. 
In particular, 
$$
\rho(X) = \rho(Y) +1,
$$ 
and 
$R$ and $f=\mathrm{cont}_R$ are classified as follows \cite{NefM82}:
\begin{itemize}
\item 
$\dim(Y) = 3$: 
then there exists an irreducible reduced divisor $D$ such that $f |_{X\setminus D}$ is an isomorphism and $\dim f(D) \leq 1$. The divisor $D$ is called the exceptional divisor, and $X$ is the blow up of $Y$ along $f(D)$ (given the reduced structure).
\begin{itemize}
\item type $E_1$ : $f(D)$ is a smooth curve, $Y$ is smooth and $f|_D : D \ra f(D)$ is a $\bP^1$-bundle, $\mu_R = 1$, and $l_R$ is a fiber of $\bP^1$-bundle $f|_D$;
\item type $E_2$ : $f(D)$ is a point, $Y$ is smooth, $D \cong \bP^2$ and $\mathcal O_D(D) \cong \mathcal O(-1)$, $\mu_R =2$, and $l_R$ is a line in $D$;
\item type $E_3$ : $f(D)$ is an ordinary double point, $D\cong \bP^1 \times \bP^1$, $\mathcal O_D(D) \cong \mathcal O(-1,-1)$, $\mu_R=1$ and $l_R$ is a ruling in $D$;
\item type $E_4$ : $f(D)$ is a double point, $D$ is an irreducible reduced singular quadric in $\bP^3$, $\mathcal O_D(D) \cong \mathcal O_D \otimes O(-1)$, $\mu_R=1$, and $l_R$ is a generator of the cone $D$;
\item type $E_5$ : $f(D)$ is a quadruple point of $Y$, $D\cong \bP^2$, $\mathcal O_D(D) \cong \mathcal O(-2)$, $\mu_R=1$, and $l_R$ is a line in $D$.
\end{itemize}
\item $\dim(Y) =2$: 
then $Y$ is a smooth projective rational surface \cite[Proposition 3.5]{mori-mukai2}, 
$f: X\ra Y$ is a conic bundle, and we denote the discriminant locus by $\Delta_f$. We have
\begin{itemize}
\item type $C_1$ : $\Delta_f$ is non-empty, $\mu_R=1$, and $l_R$ is an irreducible component of a reducible fiber or a reduced part of a multiple fiber or
\item type $C_2$ : $f$ is the projective bundle of a rank two vector bundle, $\mu_R=2$ and $l_R$ is a fiber of $f$.
\end{itemize}
\item  
$\dim(Y)=1$:
then $Y\simeq \bP^1$ and $\rho(X) = 2$. Every fiber of $f$ is irreducible and reduced and the generic fiber $X_\eta$ is a del Pezzo surface. There are three cases:
\begin{itemize}
\item type $D_1$ : $X_\eta$ is a del Pezzo surface of degree $d$, $1 \leq d \leq 6$, $\mu_R = 1$, and $l_R$ is a line on a fiber;
\item type $D_2$ : $f$ is a quadric bundle, $\mu_R = 2$, and $l_R$ is a line in a fiber;
\item type $D_3$ : $f$ is a $\bP^2$-bundle, $\mu_R=3$, and $l_R$ is a line in a fiber.
\end{itemize}
\end{itemize}
Suppose that $X$ is a Fano threefold of $\rho(X) = 2$. The cone of effective curves $\mathrm{NE}_1(X)$ consists of two rays $R_1$ and $R_2$. Set $f_i = \mathrm{cont}_{R_i} : X \ra Y_i$, $\mu_i = \mu_{R_i}$ and $l_i = l_{R_i}$. Let $L_i$ be the pull back of the ample generator on $Y_i$. The following theorem plays a central role in our analysis:
\begin{theo}{\cite[Theorem 5.1]{mori-mukai2}}
The set $\{L_1, L_2\}$ forms a basis for $\Pic (X)$ and $\{l_2, l_1\}$ is the dual basis of $\mathrm{N}_1(X)_\bZ$. The cone $\mathrm{NE}_1(X)$ is generated by $l_1$ and $l_2$. Moreover we have
$$
-K_X \sim \mu_2L_1 + \mu_1L_2.
$$
\end{theo}

We will also use the following results:

\begin{lemm} \label{lemm:normalityofconicbundle}
Let $X$ be a smooth threefold admitting the structure of a conic bundle $\pi: X \to \mathbb{P}^{2}$.  Let $C \subset \mathbb{P}^{2}$ be a line not contained in the discriminant locus for $\pi$.  Then $\pi^{-1}(C)$ is normal.
\end{lemm}

\begin{proof}
Since $X$ is a conic bundle, we have an inclusion $X \subset \mathbb{P}_{\mathbb{P}^{2}}(\mathcal{E})$ into a $\mathbb{P}^{2}$-bundle over $\mathbb{P}^{2}$.

Let $S$ denote the preimage of $C$.  Note that $S$ is irreducible by assumption on $C$.  Since $S$ is a hypersurface, it suffices by Serre's criterion to show that the singularities of $S$ have codimension at least $2$.  Again by assumption on $C$, the only way $S$ can have non-isolated singularities is if it contains a double line fiber of $\pi$.

Choose a sufficiently small open neighborhood $U \subset \mathbb{A}^{2}$ of the image of this fiber that trivializes the bundle $\mathbb{P}_{\mathbb{P}^{2}}(\mathcal{E})$.  We may assume the point of the double fiber is given by $x=y=0$ and the line $C$ is defined by $x=0$.  $X$ is defined locally by the vanishing of
\begin{equation*}
F(x,y,t_{0},t_{1},t_{2}) = \sum f_{i,j}t_{i}t_{j},
\end{equation*}
where the $f_{i,j}$ are rational functions in $x$ and $y$.  Along the fiber over $(0,0)$, we have $\frac{dF}{dy} = \frac{dF}{dt_{i}} = 0$ for each $i$.  But then the scheme $\frac{dF}{dx} = 0$ meets with the double line, so $X$ itself is singular, a contradiction.
\end{proof}

\begin{lemm} \label{lemm:singularitiesofp1fibration}
Let $S$ be a normal Gorenstein surface admitting a morphism to a curve $\pi: S \to \mathbb{P}^{1}$ such that the general fiber of $\pi$ is isomorphic to $\mathbb{P}^{1}$.  Then $S$ has at worst canonical singularities.
\end{lemm}

\begin{proof}
We first recall some facts about singularities of surfaces.  Let $\widetilde{S}$ be a germ of a normal Gorenstein surface singularity and 
$\widetilde{\phi}: \widetilde{Y} \to \widetilde{S}$ a minimal resolution.    Suppose that the singularity of $\widetilde{S}$ is not canonical.  \cite[Proposition 4.20]{Reid97} shows that we have $-K_{\widetilde{Y}/\widetilde{S}} \geq Z_{num}$, where $Z_{num}$ is the numerical cycle for the singularity.  Furthermore, \cite[Theorem 3]{Artin66} shows that $p_{a}(Z_{num}) \geq 1$ (since by assumption the singularity is Gorenstein but not canonical, and hence not rational).

Returning to our original situation, consider a singular point $s$ of $S$.  Let $Y$ denote a minimal desingularization of $S$.  Let $E$ be the part of $-K_{Y/S}$ lying over this singularity, and $E_{num}$ the numerical cycle over this point.  We have an exact sequence
\begin{equation*}
H^{1}(\mathcal{O}_{Y}) \to H^{1}(\mathcal{O}_{E_{num}}) \to H^{2}(-E_{num}).
\end{equation*}
Since $Y$ is rational, it has irregularity $H^{1}(\mathcal{O}_{Y}) = 0$.  The last space is dual to $H^{0}(K_{Y} + E_{num})$; since this divisor has negative intersection with the class of a fiber, the space is again $0$.  So we see that $H^{1}(\mathcal{O}_{E_{num}}) = 0$, showing that $p_{a}(E_{num}) \leq 0$.  Thus every singularity is canonical.
\end{proof}

A typical application of these lemmas is:

\begin{coro} \label{coro:aconditionforlines}
Let $X$ be a smooth Fano threefold
admitting a conic bundle structure $\pi: X \to \mathbb{P}^{2}$. 
Then, for any line $C\subset \mathbb{P}^{2}$ not contained in the discriminant locus, the divisor $-K_{X}$ is balanced with respect to $\pi^{-1}(C)$.
\end{coro}

\begin{proof}
Note that $a(X,-K_{X})=1$.  Also, since $X$ admits a morphism to $\mathbb{P}^{2}$ we have $b(X,-K_{X}) = \dim \NS(X,\mathbb{R}) \geq 2$.  

Lemma \ref{lemm:normalityofconicbundle} shows that $\pi^{-1}(C)$ is normal irreducible (and hence Gorenstein).  Applying Lemma \ref{lemm:singularitiesofp1fibration} we see that $\pi^{-1}(C)$ has only canonical singularities.  Note that $K_{\pi^{-1}(C)} + a(X,-K_{X})(-K_{X})$ has the class of a fiber $F$ of the map to the line $C$.  Since the relative canonical divisor of a resolution of $\pi^{-1}(C)$ is trivial, we conclude by Proposition \ref{prop:surfaceconstants} (1) that $b(\pi^{-1}(C),-K_{X}) = 1$. 
\end{proof}

Suppose that $X$ is a primitive Fano threefold of $\rho(X) =2$. There are $9$ deformation types, 
and possible values of the degree $d(X)$ are
$$
d(X) \in \{6,12,14,24,30,48,54,56,62\}.
$$
A list of these can be found in \cite[Theorem 1.7]{mori-mukai2}.

\subsection{$\rho(X) =2$ and $d(X) = 62$}
\label{subsubsec: two 62}

\

Then $X\simeq \bP(\mathcal O_{\bP^2} \oplus \mathcal O_{\bP^2}(2))$, 
the type of $(R_1 - R_2)$ is $(C_2 - E_5)$, and there is a divisorial contraction 
$$
f_2 : X \ra Y_2,
$$
with exceptional divisor $D\simeq \bP^2$ and a
$\bP^1$-bundle $f_1 : X \ra \bP^2$ with section $D$  (\cite[Lemma 4.5]{ott}). 
The intersection numbers are given by
$$
L_1^3 = 0, \quad L_1^2\cdot L_2 = 1, \quad L_1\cdot L_2^2 = 2, \quad L_2^3 =4.
$$ 

\begin{prop}
\label{prop: two 62}
$-K_X$ is balanced, with exceptional set $D$.
\end{prop}
\begin{proof}
Put $L:=-K_X$ and let $C$ be any irreducible reduced curve on $X$. 
Then $C\equiv nl_1 +ml_2$, with $n, m \in \bZ_{\geq 0}$, hence
$$
L\cdot C = (L_1 + 2L_2)\cdot (nl_1+ml_2) = 2n+m.
$$
We conclude that $L$-lines are lines in $D \cong \bP^2$, and $L$-conics are conics in $D$ and fibers of $f_1$. 
In particular, $L$ is balanced with respect to any curve $C$ not contained in $D$. Note that $b(X,L)=2$.

Suppose that $S$ is an irreducible reduced surface which is not equal to $D$. 
Let $\beta : \tilde{S} \ra S$ be a resolution of singularities of $S$. 
Since $S$ is irreducible and reduced and is not equal to $D$, 
$S$ is linearly equivalent to $nL_1 + mL_2$ where $n, m \in \bZ_{\geq0}$. Therefore,
$$
\hskip 1cm (\beta^*L|_{\tilde{S}})^2 = L^2\cdot S = (L_1+2L_2)^2\cdot (nL_1+mL_2) = 12n+25m.
$$
By Theorem~\ref{theo: Reider}, $\beta^*L|_{\tilde{S}} + K_{\tilde{S}}$ is effective and
$$
a(S, L) \leq a(X, L) =1.
$$ 
If $a(S, L) = a(X, L)$, then Theorem~\ref{theo: Reider II} implies 
that the Iitaka dimension of $\beta^*L|_{\tilde{S}} + K_{\tilde{S}}$ is $1$. 
We conclude that $b(S, L) =1$ and $L$ is balanced with respect to $S$. 
\end{proof}

\subsection{$\rho(X)=2$ and $d(X) = 56$} 

\

Then
$X\simeq \bP(\mathcal O_{\bP^2} \oplus \mathcal O_{\bP^2}(1))$,  
the type of $(R_1-R_2)$ is $(C_2-E_2)$, and it admits a divisorial contraction 
$$
f_2 : X \ra Y_2,
$$
with exceptional divisor $D\simeq \bP^2$. 
We have $\bP^1$-bundle $f_1 : X \ra \bP^2$, and 
$D$ is a section of this fibration \cite[Lemma 4.5]{ott}. 
The intersection numbers are given by
$$
L_1^3 = 0, \quad L_1^2\cdot L_2 = 1, \quad L_1\cdot L_2^2 = 1, \quad L_2^3 =1.
$$

\begin{prop}
\label{prop: two 56}
$-K_X$ is balanced, with exceptional set $D$.
\end{prop}
\begin{proof}
We have $-K_{X} = 2L_{1} + 2L_{2}$.  Thus for any rational curve $C$ we have 
$$
a(C,-K_{X}) \leq 1 = a(X,-K_{X})\quad \text{  and } \quad 
b(C,-K_{X}) = 1 < b(X,-K_{X}).
$$  
So $X$ is balanced with respect to arbitrary curves.  The rest of the proof is similar to Proposition~\ref{prop: two 62}.\end{proof}

\subsection{$\rho(X)=2$ and $d(X) = 54$} 
\

Then $X\simeq \bP^2 \times \bP^1$ and
the type is $(C_2 - D_3)$. The intersection numbers are given by
$$
L_1^3 = 0, \quad L_1^2\cdot L_2 = 1, \quad L_1\cdot L_2^2 = 0, \quad L_2^3 =0.
$$
\begin{prop}
\label{prop: two 54}
$-K_X$ is balanced, with empty exceptional set.
\end{prop}

\begin{proof}
The proof is similar to before.  The only change is that we need to consider the divisors of class $L_{1}$ and $L_{2}$ with an additional (easy) argument.
\end{proof}

\subsection{$\rho(X)=2$ and $d(X)=48$} 

\

Then $X\subset \bP^2 \times \bP^2$ is a nonsingular divisor of bidegree $(1,1)$ and
the type of $(R_1-R_2)$ is $(C_2-C_2)$. The intersection numbers are given by
$$
L_1^3 = 0, \quad L_1^2\cdot L_2 = 1, \quad L_1\cdot L_2^2 = 1, \quad L_2^3 =0.
$$
\begin{prop}
\label{prop: two 48}
$-K_X$ is balanced, with empty exceptional set.
\end{prop}

\begin{proof}
The proof is similar to before.
\end{proof}

\subsection{$\rho(X)=2$ and $d(X)=30$} 

\

Then $X\subset \bP^2\times\bP^2$ is a nonsingular divisor 
of bidegree $(1,2)$ and 
the type of $(R_1-R_2)$ is $(C_1-C_2)$. The intersection numbers are given by
$$
L_1^3 = 0, \quad L_1^2\cdot L_2 = 2, \quad L_1\cdot L_2^2 = 1, \quad L_2^3 =0.
$$
Let $\Delta_{f_1}$ be the discriminant locus of the first projection $f_1 : X \ra \bP^2$. 
Write the union of singular fibers of $f_1$ by $Z(X)$. 
\begin{prop}
\label{prop: two 30}
$-K_X$ is balanced, with exceptional set $Z(X)$.
\end{prop}
\begin{proof}
Let $L$ be the anticanonical class $-K_X$. Suppose that $C$ is an irreducible reduced curve on $X$. Then $C$ is numerically equivalent to $nl_1+ml_2$ where $n, m \in \bZ_{\geq0}$. Thus we have
$$
-K_X\cdot C = (2L_1+L_2)\cdot (nl_1+ml_2) = n+2m.
$$
We conclude that $Z(X)$ is the surface swept out by $L$-lines, and $L$-conics are fibers of $f_1$ and $f_2$. In particular, $L$ is balanced with respect any curve $C$ not contained in $Z(X)$.

Let $S$ be an irreducible and reduced surface on $X$ not contained in $Z(X)$. We choose a smooth resolution $\beta: \tilde{S} \ra S$. Since $L_1$ and $L_2$ generate the cone of pseudo-effective divisors, $S$ is linearly equivalent to $nL_1+mL_2$ where $n, m \in \bZ_{\geq0}$. Then we have
$$
(\beta^*L|_{\tilde{S}})^2 = L^2\cdot S = 9n+12m.
$$
Theorem~\ref{theo: Reider} shows that $\beta^*L|_{\tilde{S}} + K_{\tilde{S}}$ is effective. We conclude that $a(S, L) \leq a(X, L)=1$. Suppose that $a(S, L) = a(X, L)$. If $(n,m) \neq (1,0)$, then it follows from Theorem~\ref{theo: Reider II} that $\beta^*L|_{\tilde{S}} + K_{\tilde{S}}$ has Iitaka dimension $1$. Hence $b(S,L)=1$ and $L$ is balanced with respect to $S$. If $(n,m) = (1,0)$, then $S$ is the pullback $f_1^{-1}(l)$ of a line $l$ on $\bP^2$ not contained in $\Delta_{f_1}$.  
This case is treated by Corollary~\ref{coro:aconditionforlines}.
\end{proof}

\subsection{$\rho(X)=2$ and $d(X) =24$} 

\

Then $X\ra \bP^2\times \bP^1$ is a double cover with branch locus 
a divisor of bidegree $(2,2)$. The type of $(R_1-R_2)$ is $(C_1-D_2)$. The intersection numbers are given by
$$
L_1^3 = 0, \quad L_1^2\cdot L_2 = 2, \quad L_1\cdot L_2^2 = 0, \quad L_2^3 =0.
$$
We define $Z(X)$ as the union of singular fibers of $f_1$.
\begin{prop}
\label{prop: two 24}
$-K_X$ is weakly balanced, with exceptional set $Z(X)$, but not balanced.
\end{prop}
\begin{proof}
Write $L$ for $-K_X$. Smooth fibers $S$ of $f_2$ are smooth quadrics, so we have $a(S, L)=1$ and $b(S, L)=2$. Thus $-K_X$ is not balanced. 

Let $C$ be an irreducible reduced curve on $X$. Then $C$ is numerically equivalent to $nl_1+ml_2$ where $n, m \in \bZ_{\geq0}$. It follows that
$$
-K_X\cdot C = (2L_1 + L_2)\cdot(nl_1+ml_2) = n+2m.
$$
Thus $L$-lines are irreducible components of singular fibers of $f_1$, and $L$-conics are general fibers of $f_1$ and rulings of fibers of $f_2$. We conclude that $L$ is balanced with respect to any curve $C$ not contained in $Z(X)$. 

Let $S$ be an irreducible reduced surface on $X$. We choose a smooth resolution $\beta : \tilde{S} \ra S$. Since $L_1$ and $L_2$ generate the cone of pseudo-effective divisors, $S$ is linearly equivalent to $nL_1+mL_2$ where $n,m \in \bZ_{\geq 0}$. Therefore we have
$$
(\beta^*L|_{\tilde{S}})^2 = L^2.S = 8n+8m.
$$
By Theorem~\ref{theo: Reider}, $\beta^*L|_{\tilde{S}} + K_{\tilde{S}}$ is effective, so 
$$
a(S, L) \leq a(X, L) =1.
$$ 
If $(n, m) \neq (1,0), (0,1)$, then Theorem~\ref{theo: Reider II} shows that $\beta^*L|_{\tilde{S}} + K_{\tilde{S}}$ has Iitaka dimension $1$ so that $b(S, L) =1$. We conclude that $L$ is balanced with respect to $S$. Suppose that $(n, m)=(0,1)$. Then $S$ is a fiber of $f_2$ and $S$ is either a smooth quadric or a singular quadric cone in $\bP^3$. It follows that $L$ is weakly balanced with respect to $S$. The case of $(n,m) = (0,1)$ is treated by Corollary~\ref{coro:aconditionforlines}.

\end{proof}

\subsection{$\rho(X)=2$ and $d(X)=14$} 

\

Then $f: X\ra V_7:=\bP(\mathcal O_{\bP^2} \oplus \mathcal O_{\bP^2}(1))$ 
is a double cover,  
with branch locus a member of $-K_{V_7}$. The type of $(R_1-R_2)$ is $(C_1-E_3)$ or $(C_1-E_4)$. 
Thus we have a divisorial contraction 
$$
f_2 : X \ra Y_2.
$$
We denote its exceptional divisor by $D$. We also have the conic fibration $f_1 : X \ra \bP^2$, and the restriction of $f_1$ to $D$ is a double cover (\cite[Lemma 4.5]{ott}). 
The intersection numbers are given by
$$
L_1^3 = 0, \quad L_1^2\cdot L_2 = 2, \quad L_1\cdot L_2^2 = 2, \quad L_2^3 =2.
$$

Let $Z(X)$ be the union of singular fibers of $f_1$ and $D$.

\begin{prop}
\label{prop: two 14}
$-K_X$ is balanced, with exceptional set $Z(X)$.
\end{prop}

\begin{proof}
Let $L$ be the anticanonical class $-K_X$. It is easy to verify that $L$ is balanced with respect to any curve $C$ not contained in $Z(X)$. Suppose that $S$ is an irreducible and reduced surface on $X$ not contained in $Z(X)$. Let $\beta : \tilde{S} \ra S$ be a smooth resolution of $S$. Since $S$ is not equal to $D$, $S$ is linearly equivalent to $nL_1+mL_2$ where $n,m \in \bZ_{\geq 0}$. Then we have
$$
(\beta^*L|_{\tilde{S}})^2 = L^2\cdot S = 6n+8m.
$$
It follows from Theorem~\ref{theo: Reider} that $a(S, L) \leq a(X, L)=1$. Suppose that $a(S,L) = a(X, L) =1$. First we assume that $(n,m) \neq (1,0), (0,1)$. If the Iitaka dimension of $\beta^*L+K_{\tilde{S}}$ is zero, then Theorem~\ref{theo: Reider II} shows that any general pair $x, y$ can be connected by a $L$-conic $D$ with $D^2=0$.  Note that if we fix a general point $x$ and choose any two general points $y,y'$, the corresponding conics $D,D'$ satisfy $D \cdot D' > 0$ but $D^{2} = D'^{2} = 0$.  Thus, there are infinitely many numerical classes of curves on $S$ with bounded $-K_{X}|_{S}$-degree, which is a contradiction. This shows that the Iitaka dimension of $\beta^*L + K_{\tilde{S}}$ is $1$, and we conclude that $b(S, L) =1$. The divisor $L$ is balanced with respect to $S$. 

Next suppose that $S$ is an irreducible surface of class $(1,0)$.  Corollary \ref{coro:aconditionforlines} shows that 
$-K_{X}$ is balanced with respect to $S$.

Finally, suppose that $S$ is an irreducible surface of class $(0,1)$.  The variety $Y_2$ is a double cover of $\bP^3$ ramified along a quartic with one singular point. Then the image $H$ of $S$ under the morphisms to $\bP^3$ is a hyperplane in $\mathbb{P}^{3}$.  Thus $S$ is a double cover of $\mathbb{P}^{2}$ ramified over a quartic, which avoids the singular point.
One can prove that $S$ is normal as Lemma~\ref{lemm: isolated del Pezzo of degree two}.


The analysis of singularities in \cite[Chapter 4]{KM98} shows that $S$ has at worst elliptic singularities, and has at worst canonical singularities unless the quartic defining the branch locus has a singularity of multiplicity 4.  If $S$ has canonical singularities, then since $S|_{S}$ is big we see that $a(S,-K_{X}) < a(X,-K_{X})$.

Assume we are in the elliptic singularity case.  Write $f(x,y,z)$ for the quartic defining the branch locus.  By a coordinate change we may assume a singularity of multiplicity 4 occurs at $(x,y,z) = (1,0,0)$, so that in fact $f(x,y,z) = f(y,z)$ only depends on two variables.  Then we can realize $S$ as the hypersurface $w^{2} = f(y,z)$ in $\mathbb{P}(1,1,1,2)$.  This shows that $S$ has at most one elliptic singularity.  Note that in this equation $f(y,z)$ can not have any multiple roots since $S$ is normal.

Let $\phi: S' \to S$ be a weighted blow-up of the cone point.  Then $S'$ admits a morphism to the curve $w^{2} = f(y,z)$ with general fiber isomorphic to $\mathbb{P}^{1}$.  Since $f$ has no multiple roots this equation defines an elliptic curve.  Thus $S$ is not rational and we conclude that $b(S,-K_{X}) = 1$ by Proposition \ref{prop:surfaceconstants}.  Since we have already verified the inequality of $a$-constants, we see that $-K_{X}$ is balanced with respect to $S$. 
\end{proof}

\subsection{$\rho(X)=2$ and $d(X)=12$} 

\

Let $W_6\subset \bP^2 \times \bP^2$ be a nonsingular divisor of bidegree $(1,1)$. 
Then $X$ is either isomorphic to a double cover of $W_6$ whose branch locus is a member of $|-K_{W_6}|$, 
or a nonsingular divisor on $\bP^2 \times \bP^2$ of bidegree $(2,2)$. The type of $(R_1-R_2)$ is $(C_1-C_1)$. 
The intersection numbers are given by
$$
L_1^3 = 0, \quad L_1^2\cdot L_2 = 2, \quad L_1\cdot L_2^2 = 2, \quad L_2^3 =0.
$$
Let $Z(X)$ be the union of singular fibers of $f_1$ and $f_2$. Then we have
\begin{prop}
\label{prop: two 12}
$-K_X$ is balanced, with exceptional set $Z(X)$.
\end{prop}

\begin{proof}
Write $L$ for $-K_X$. It is easy to check that $L$ is balanced with respect to any curve $C$ not contained in $Z(X)$. Suppose that $S \subset X$ is an irreducible and reduced surface. Let $\beta:\tilde{S} \ra S$ be a resolution of singularities. Then $S$ is linearly equivalent to $nL_1 + mL_2$ where $n,m \in \bZ_{\geq 0}$. We have
$$
(\beta^*L|_{\tilde{S}})^2 = L^2\cdot S = 6n + 6m.
$$
By Theorem~\ref{theo: Reider}, $a(S, L) \leq a(X, L)=1$.
If $a(S, L) = a(X, L)=1$ and $(n,m) \neq (1,0), (0,1)$, then by arguing just as in the proof of Proposition \ref{prop: two 14}, Theorem~\ref{theo: Reider II} shows that $b(S, L) =1$. The remaining cases $(n,m) = (1,0), (0,1)$ are handled by Corollary \ref{coro:aconditionforlines}.
\end{proof}

\subsection{The case of $\rho(X)=2$ and $d(X)=6$} 

\

Then $X\ra \bP^2 \times \bP^1$ is a double cover, with branch locus 
a divisor of bidegree $(4,2)$. The type of $(R_1-R_2)$ is $(C_1-D_1)$. The intersection numbers are given by
$$
L_1^3 = 0, \quad L_1^2\cdot L_2 = 2, \quad L_1\cdot L_2^2 =0 , \quad L_2^3 =0.
$$
Let $Z(X)$ be the union of singular fibers of $f_1$ and $f_2$, and lines in general fibers of $f_2$. Then we have

\begin{prop}
\label{prop: two 6}
$-K_X$ is weakly $a$-balanced, with $a$-ex\-ceptional set $Z(X)$, but not weakly balanced.
\end{prop}

\begin{proof}
Write $L$ for $-K_X$. A general fiber $S$ of $f_2$ is a del Pezzo surface of degree $2$. Thus 
$a(S, L) =1$ and $b(S, L) =8$. The anticanonical class $-K_X$ is not weakly balanced. It is easy to verify that $L$ is balanced with respect to any curve not contained in $Z(X)$. Suppose that $S$ is an irreducible and reduced surface on $X$ not contained in $Z(X)$. Let $\beta:\tilde{S} \ra S$ be a resolution of singularities of $S$. Since $S$ is linearly equivalent to $nL_1+mL_2$ where $n,m \in \bZ_{\geq0}$, we have
$$
(\beta^*L|_{\tilde{S}})^2 = L^2\cdot S = 4n+2m.
$$
If $(n,m) \neq (1,0), (0,1)$, then by Theorem~\ref{theo: Reider}, $a(S, L) \leq a(X, L)=1$. The case $(n,m) = (0,1)$ was already discussed.  The remaining case was treated by Corollary~\ref{coro:aconditionforlines}.

\end{proof}

We summarize our results for primitive Fano threefolds of $\rho(X)=2$:

\begin{theo}
\label{thm:fano-rank2}
Let $X$ be a primitive Fano threefold of $\rho(X)=2$;
\begin{itemize}
\item if $d(X) = 12, 14, 30, 48, 54, 56, 62$, then $-K_X$ is balanced;
\item if $d(X) =24$, then $-K_X$ is weakly balanced, but not balanced;
\item if $d(X) = 6$, then $-K_X$ is weakly $a$-balanced, but not weakly balanced.
\end{itemize}
\end{theo}


\section{Arithmetic applications}
\label{sect:applications}

In this section we discuss applications of the theory of balanced line bundles to 
counting problems, specifically to Manin's conjecture and its generalizations 
formulated in Section~\ref{secct:intro} (see \eqref{eqn:manin}). We work over a number field $F$, focus on smooth Fano varieties, and restrict 
our attention to $L=-K_X$.

In Section~\ref{sect:rigidity} 
we showed that there exists a Zariski closed subset 
$V\subset X$ such that all subvarieties $Y\subsetneq X$ with $a(Y, -K_X)>1$ are
contained in $V$: Theorem~\ref{theo:aclosedness} establishes this
under the WBAB$_{n-1}$ conjecture, if $\dim(X)\ge 5$, 
and unconditionally if $\dim(X)\le 4$.

Put $X^\circ:=X\setminus V$. 
Thus, a first approximation to Manin's conjecture would be 
$$
N(X^{\circ}(F'), -\mathcal{K}_X, B) \ll B^{1+\epsilon}, \quad B\ra \infty,
$$
for all finite extensions $F'/F$ and all $\epsilon >0$. 
This is an open problem even for smooth cubic surfaces, but at least, 
we know of no counterexamples. As for lower bounds, currently, even potential density 
of rational points, i.e., Zariski density after a finite field extension, is an open 
problem.

The detailed analysis of the geometry of Fano threefolds in 
Sections~\ref{sect:fano-threefolds} and  \ref{sect:seven} shows that, 
{\em geometrically},  
there may exist dominating families  $\mathcal Y\ra B$  
of subvarieties $Y_b\subsetneq X$, 
parametrized by a base $B$, such that for each $b$ in some 
Zariski open subset $B^\circ\subset B$, 
$$
a(Y_b, -K_X)=1, \quad \text{ but} \quad b(Y_b,-K_X)\ge b(X,-K_X). 
$$
Such families are sources of counterexamples to Manin's conjecture 
and its refinement by Peyre concerning 
the shape of the leading constant in \eqref{eqn:manin}. 
The philosophy of MMP 
(finite generation of ample and effective cones in the Fano case) 
suggests that there should be finitely many such families.

For arithmetic applications, the following issue is apparent:  
\begin{itemize}
\item The parametrizing families could have arbitrary Kodaira dimension: 
some of the bases $B$ satisfy potential density of rational 
points, while others are of general type, and have few rational points.  
\end{itemize}

We will now present several examples 
from the literature and from Section~\ref{sect:seven}, 
which illustrate this issue:
\begin{itemize}
\item There exist dominating families $\cY\ra B$ of cubic surfaces with  
$$
a(Y_b, -K_X)=1, \quad \text{ but} \quad b(Y_b,-K_X) > b(X,-K_X)=2, 
$$
for a Zariski dense set of $b\in B(F)$, see Example~\ref{exam:maninconjfail}.
\item    
There exist dominating families $\cY\ra B$ with 
$$
a(Y_b, -K_X)=1, \quad \text{ but} \quad b(Y_b,-K_X) = b(X,-K_X)\le 2, 
$$
for a Zariski dense set of $b\in B(F)$, e.g., quadric bundles in  
Example~\ref{exam:quadraticbundle} and Example~\ref{exam:tamagawaconstantfail}.
Moreover, numerical experiments in Example~\ref{exam:quadraticbundle} 
suggest that the leading constant 
$c(-\mathcal K_X)$ is a {\em sum} of a Tamagawa type constant, as predicted 
by Peyre, and constants arising from $b\in B(F)$ parametrizing split fibers.  
\item 
There exist dominating families $\cY\ra B$ of quadrics with
$$
a(Y_b, -K_X)=1, \quad \text{ but} \quad b(Y_b,-K_X) = b(X,-K_X)=2, 
$$
but only for finitely many $b\in B(F')$, for all finite extensions $F'/F$, 
see Example~\ref{exam:2}.
\item There exist dominating families $\cY\ra B$ 
of degree two Del Pezzo surfaces with
$$
a(Y_b, -K_X)=1, \quad \text{ but} \quad b(Y_b,-K_X) <  b(X,-K_X)=2, 
$$
for all but finitely many $b\in B(F')$, for all finite extensions $F'/F$, 
while, {\em geometrically}, each fiber $Y_b$ has $b(Y_b,-K_X)=8$. 
\end{itemize}

\begin{exam}\cite{BT-cubic} \label{exam:maninconjfail}
Consider the Fano fivefold
\begin{equation*}
X = \left\{ \sum_{i=0}^{3} x_{i}y_{i}^{3} = 0 \right\} \subset \mathbb{P}^{3}_{x_{0},x_{1},x_{2},x_{3}} \times \mathbb{P}^{3}_{y_{0},y_{1},y_{2},y_{3}}.
\end{equation*}
The first projection $\pi: X \to \mathbb{P}^{3}$ fibers $X$ in cubic surfaces.  Let $Y$ denote a fiber of $\pi$.  We have
\begin{equation*}
a(X,-K_{X}) = 1 = a(Y,-K_{X}), \qquad \text{ and } \quad  b(X,-K_{X}) = 2.
\end{equation*} 
The Picard rank of $Y$ can vary depending on the ground field.  
However, \cite{BT-cubic} shows that over any field containing $\sqrt{-3}$, 
a Zariski dense set 
of fibers attains the maximal value $b(Y,-K_{X}) = 7$.  
Thus, the behavior of the anticanonical 
counting function on $X$ is controlled by the invariants of the cubic surfaces instead of the invariants for $X$.
\end{exam}

\begin{exam}\cite{quadratic11}
\label{exam:quadraticbundle}
Let $X$ be a smooth threefold defined by
$$
ax^2+by^2+l_1(a,b)z^2 + l_2(a,b)w^2 = 0,
$$
in $\bP^1 \times \bP^3$, where $l_1$ and $l_2$ are linear forms in $a,b$. This is a Fano threefold which admits a fibration into quadric surfaces. 
The anticanonical class $-K_X$ is not balanced with respect to split 
quadrics, and there are infinitely many such fibers if the elliptic curve
$$
u^2 = abl_1(a,b)l_2(a,b)
$$
has positive Mordell-Weil rank and there is at least one split fiber. Thus Peyre's constant should not govern the growth of rational points on $X$.   
In \cite{quadratic11}, Elsenhans conducted a 
experiment and observed that after 
removing contributions from split fibers, 
the leading constant converges to Peyre's constant, at least numerically.
\end{exam}

\begin{exam}\cite{BL13} \label{exam:tamagawaconstantfail}
Let $X$ be a complete intersection of two quadrics in $\mathbb{P}^{5}$; 
it is covered by lines $C$; 
the $a$ and $b$ invariants coincide for $X$ and for $C$.  
However, the Tamagawa-type constants do not coincide, and in fact the asymptotics are not controlled by the Tawagama-type constant for $X$.  However, as discussed in \cite[Conjecture 1.4]{BL13} there is the possibility that if one removes the contributions to the point count coming from thin subsets then the original version of Manin's conjecture is still valid.
\end{exam}

\begin{exam}
\label{exam:2}
Let $X$ be a double cover of $\mathbb P^2 \times \mathbb P^1$ ramified along a smooth divisor of bidegree $(2,2)$. The variety $X$ is a primitive Fano theefold of $\rho(X)=2$ and $d(X) =24$. The projection $\pi:X\ra \bP^1$ 
is a quadric surface bundle. 
Consider the variety $\mathcal W \ra \bP^1$ of relative lines for $\pi$. Take the Stein factorization $\mathcal W \ra C \ra \bP^1$. Then $C$ is a double cover of $\bP^1$. When $X$ is general, $C$ is smooth and irreducible. Moreover, $C$ is ramified at six points. Hence $C$ is a curve of genus two. By Faltings' theorem, 
there are only finitely many rational points on $C$, thus there are only finitely many split fibers of $\pi$.

We can expect that for a fixed number field, after removing finitely many split fibers, 
the growth rate of heights of rational points is as predicted by 
Manin's Conjecture (with Peyre's constant).  
However, Manin's Conjecture predicts the behavior over all extensions as well.  
Since any fixed Zariski open $X^{\circ}\subset X$ will intersect split fibers over some extension, 
we should expect to find ``extra'' contributions to the constant $c(X^{\circ},-\mathcal{K}_X)$ arising from the fibers.
\end{exam}

\begin{exam}
\label{exam:3}
Let $X\ra \bP^2 \times \bP^1$ be a double cover 
ramified along a smooth divisor of bidegree $(4,2)$. It is 
a primitive Fano threefold of $\rho(X)=2$ and $d(X)=6$. 
We assume that $X$ is general enough so that  the fibers of $\pi$ admit at most one $A_1$-singularity. 
The projection $\pi:X\ra\bP^1$ is a degree $2$ del Pezzo surface fibration, 
and we consider the variety $f :  C \ra \bP^1$ of relative lines for $\pi$.  
We claim that $C$ is a smooth irreducible curve.  
Interpreting $C$ as the relative Hilbert scheme of lines, a normal bundle calculation shows that $C$ is smooth. 
Let $U$ be the complement of the discriminant hypersurface in the space of plane quartics. 
The variety $X$ induces a conic $F$ in the space of plane quartics, 
and it follows from a Lefschetz theorem that for a general $F$ 
$$
\pi_1(F\cap U) \ra \pi_1(U)
$$
is surjective. The monodromy group of $\pi_1(U)$ on a plane quartic was computed in \cite{universalhypersurface}, 
it is the symplectic group of the first cohomology of the plane quartic. 
There is a one-to-one correspondence between smooth nonhyperelliptic curves of genus $3$ 
with Aronhold bases and del Pezzo surfaces $S$ of degree $2$ with marked Picard group (see, e.g.,  \cite[Section 7]{GH04}). 
The symplectic group acts on the set of Arnonhold bases transitively, 
thus the monodromy acts on the $56$ lines on the degree $2$ del Pezzo surface transitively. 
We conclude that $C$ is irreducible.
The projection $f  :C \ra \mathbb P^1$ is a finite map of degree $56$. 
Since the degree of the discriminant hypersurface in the space of plane quartics is $27$,
$f$ is ramified at $54$ points. 
Any del Pezzo surface with one $A_1$-singularity contains $44$ lines, thus it follows from the Hurwitz formula that 
$$
2g(C) - 2 = 56\cdot (-2) + 54(56 - 44).
$$
This implies that $g(C) = 269$, so by Faltings' theorem there are only finitely many rational points on $C$. We further consider the variety $\mathcal W \ra \bP^1$ of relative conics for $\pi$. As the case of the variety of lines, one can show that $\mathcal W$ is smooth and irreducible. Consider the Stein factorization 
$$
\mathcal W \ra D \stackrel{h}{\longrightarrow}\bP^1.
$$
The projection $h$ has degree $126$, since
every smooth del Pezzo surface of degree two contains $126$ families of conics. 
A surface with one $A_1$-singularity contains $93$ families of conics. By Hurwitz formula, we have $g(D) = 766$. Thus for all but finitely many fibers, a fiber contains no lines and no conics defined over the ground field. 
A fiber with no lines and no conics has Picard rank $1$. 
Again, for any fixed number field, we can expect the growth of heights to behave as in 
Manin's conjecture, after removing finitely many fibers.
\end{exam}

\nocite{*}
\bibliographystyle{alpha}
\bibliography{balanced-fano4}

\end{document}